\documentclass{article}

\usepackage[a4paper,top=2.54cm,bottom=2.54cm,left=3.17cm,right=3.17cm,%
            includehead,includefoot]{geometry}

\usepackage{amsmath,amssymb,amsfonts,amsthm}
\usepackage{graphicx}
\usepackage{subfigure}
\usepackage{float}
\usepackage{xcolor,color,soul}
\usepackage[numbers,square,sort&compress]{natbib}
\usepackage{hyperref}
\usepackage{booktabs}  
  \hypersetup{colorlinks,citecolor=blue,linkcolor=blue,breaklinks=true}
\usepackage{booktabs}
\usepackage{colortbl}
\usepackage{caption}
\usepackage{enumitem}
\usepackage{epstopdf}
\usepackage{algorithm}
\usepackage{algpseudocode}
\usepackage{array}
\usepackage{bbding}
\usepackage{fancyhdr} 
\usepackage{fancyvrb} 
\usepackage{longtable}
\usepackage{listings}
\usepackage{rotating,rotfloat} 
\usepackage{yhmath} 
\usepackage{verbatim}  
\usepackage{fancyhdr}
\allowdisplaybreaks

\newtheorem{theorem}{Theorem}[section]

\newtheorem{example}{Example}
\newtheorem{remark}{Remark}

\newcommand{\bbm}{\begin{bmatrix}}
\newcommand{\ebm}{\end{bmatrix}}

\usepackage[runin]{abstract}  
\providecommand{\keywords}[1]
{
  \textbf{\text{Keywords: }} #1
}

\begin{document}
\title{\uppercase{Finite element method coupled with multiscale finite element method for the non-stationary Stokes-Darcy model}}
\author{
 Yachen Hong
    \thanks{School of Mathematical Sciences, East China Normal University, Shanghai. ychongiris@gmail.com }
  \and
  Wenhan Zhang
    \thanks{School of Mathematical Sciences, East China Normal University, Shanghai. zwh\_hanhan@163.com.
}
  \and
 Lina Zhao
  \thanks{Department of Mathematics, City University of Hong Kong, Kowloon Tong, Hong Kong SAR.}
\and
 Haibiao Zheng
\thanks{School of Mathematical Sciences, East China Normal University, Shanghai.}
}
\date{}
\maketitle
\begin{abstract}
In this paper, we combine the multiscale finite element method to propose an algorithm for solving the non-stationary Stokes-Darcy model, where the permeability coefficient in the Darcy region exhibits multiscale characteristics.
Our algorithm involves two steps: first, conducting the parallel computation of multiscale basis functions in the Darcy region. Second, based on these multiscale basis functions, we employ an implicit-explicit scheme to solve the Stokes-Darcy equations.
One significant feature of the algorithm is that it solves problems on relatively coarse grids, thus significantly reducing computational costs.
Moreover, under the same coarse grid size, it exhibits higher accuracy compared to standard finite element method.
Under the assumption that the permeability coefficient is periodic and independent of time, this paper demonstrates the stability and convergence of the algorithm. Finally, the rationality and effectiveness of the algorithm are verified through three numerical experiments, with experimental results consistent with theoretical analysis.
\end{abstract}
\numberwithin{equation}{section} 

\section{Introduction}

Water pollution is a problem that the whole world is concerned about today, which poses a serious threat to human health and the sustainable development of society. Surface water and groundwater are the main components of water resources, and they are closely related and mutually transformed. Therefore, accurate simulation of surface water and groundwater is crucial. The two form an organic whole, simulating surface water or groundwater separately may lead to insufficient accuracy. Hence, coupling the simulation of surface and subsurface would be more in line with the actual situation. We commonly use the Stokes-Darcy equations to describe the movement of fluid on the surface and subsurface, where the Stokes equation simulates the fluid flow on the surface, and the Darcy equation describes porous media flow in the subsurface.
The analysis of the coupled Stokes-Darcy model has been explored in references \cite{CGHW,DMQ,DZ,HQ,MZ,LSY,SZheng,SZL,XH}. The most commonly used model incorporates the Beavers–Joseph–Saffman interface condition(BJS) \cite{BJ,MJ,Sa}.

In practice, many industrial and scientific problems exhibit multiscale characteristics, such as the electrical conductivity of composite materials, the design of integrated circuits, and the study of fluids in porous media, and so on.
The complexity and diversity of multiscale problems have resulted in various multiscale computational methods.
For example, multiscale finite element method (MsFEM) \cite{EH,HTW,HWC,SXDJ,YDC}, generalized multiscale finite element method (GMsFEM) \cite{FEZ,WEZ},
 heterogeneous multiscale method (HMM) \cite{ELRV,EM,MuZ}, upscaling method \cite{FMV,HJ}, homogenization method \cite{MiZ,MY}, and variational multiscale method \cite{ZHSS}, and so on.

In this paper, we investigate the non-stationary Stokes-Darcy model, in which the underground aquifer controlled by the Darcy equation demonstrates heterogeneity, with the permeability coefficient exhibiting multiscale characteristics. Specifically, in this study, we consider the permeability coefficient as a periodic function in space, independent of time.
In many studies, for the sake of simplicity, it is common to assume that the underground aquifer is homogeneous, meaning the permeability coefficient remains constant. When using traditional finite element method to solve problems, the permeability coefficient is treated as a constant within each element. This assumption is reasonable for homogeneous groundwater problems.
Nevertheless, when the subsurface aquifer exhibits heterogeneity and features multiscale characteristics.
If standard finite element method is used for solving, the mesh must be sufficiently fine to ensure obtaining accurate results. However, this results in a discretized problem with a high degree of freedom, requiring a significant amount of computational memory and CPU time. It is easy to exceed the current limits of computer resources. Therefore, our aim is to conserve computational resources while still maintaining computational accuracy. From another perspective, what we are usually concerned about is the characteristics of the solution at the macroscopic scale and the influence of the microscopic scale on the macroscopic scale, rather than all the details of the solution at the microscopic scale. Therefore, we expect to obtain the influence of the microscopic scale on the macroscopic scale without delving into the details of the small scale \cite{HWC,HTW}.

In this case, we are considering the combination of multiscale finite element method (MsFEM) with the standard finite element method (FEM). Specifically, we employ the finite element method to solve the problem in the Stokes region, while applying the multiscale finite element method in the Darcy region. The multiscale finite element method, proposed by scholars such as Hou, is primarily used for addressing two-dimensional second-order elliptic boundary value problems with high oscillation coefficients \cite{HWC}. First, perform the grid partitioning, then solve the local homogeneous elliptic equation in parallel on each element to construct multiscale basis functions. Through this process, information at the microscopic scale is embedded into the basis functions, which encompass rapid oscillations. Finally, solve the original problem on coarse grid elements, during which the basis functions transmit microscopic information to the macroscopic scale, resulting in a solution at the macroscopic scale. This method not only saves computational resources but also maintains computational accuracy.

The main difference between multiscale finite element and finite element method lies in the basis functions they use. The finite element method typically uses basis functions that are linear, quadratic, or higher-order elements. These functions are often unrelated to the equations and do not contain specific information about the region. In contrast, the basis functions of multiscale finite element include microscale information about the region. This design allows multiscale finite element to maintain high computational accuracy even when solving on coarse grids. 

When solving the elliptic equation in a single region, the multiscale finite element method was used in \cite{YDC,MS,HTW,HWC}, and related research on the application of multiscale finite element method in solving parabolic equation can be found in  \cite{EH,SXDJ}. Therefore, it can be considered that the multiscale finite element method has become quite mature in the application of single region elliptic equation and parabolic equation.
Additionally, in the context of the two-region coupling problem, Zhang, Yao, Huang, and Wang utilized the multiscale finite element method to investigate the two-phase flow problem in fractured media. They experimentally demonstrated the effectiveness and accuracy of the method, but no theoretical analysis was provided \cite{ZYHW}.
In \cite{GVY}, the mortar multiscale finite element method was applied to the Stokes-Darcy model with BJS interface conditions.
In \cite{HZZZ}, we combined the multiscale finite element method with the finite element method to study the steady-state Stokes Darcy model, where the permeability coefficient of the Darcy region has multiscale characteristics, proving the $L^2$ and $H^1$ errors of the model. 

The essence of our method is the construction of basis functions for the Darcy region. In our model, the Darcy equation is a parabolic equation.
The permeability coefficient we consider is a periodic function in space, which is independent of time.
Therefore, the approach to constructing the basis functions can be aligned with that used for elliptic equation \cite{EH}. By removing the time derivative terms from the parabolic equation, it is transformed into an elliptic equation. Then, solving the local homogeneous elliptic equation allows us to obtain the basis functions.

We present the structure of our paper as follows: in Section 2, we introduce the coupled Stokes-Darcy model. Moving on to Section 3, we provide a detailed exposition of the MsFEM-ImEx algorithm for the Stokes-Darcy problem, accompanied by stability and error estimates. Finally, in Section 4, we present the numerical experiments.

\section{Coupled Stokes-Darcy model }
\subsection{Problem}
Let us consider a fluid flow in $\Omega_{f}$ coupled with porous media flow in $\Omega_{p}$, where $\Omega_{f}$, $\Omega_{p}\subset R^{d}$(d=2) are bounded domains,
$\Omega_{f}\bigcap \Omega_{p}= \emptyset$ and $\overline\Omega=\overline\Omega_{f} \bigcup \overline\Omega_{p}$. 
Denote $\Gamma=\overline\Omega_{f} \bigcap \overline\Omega_{p}$, $\Gamma_{f}=\partial \Omega_{f}\backslash\Gamma$ and $\Gamma_{p}=\partial \Omega_{p}\backslash\Gamma$.
The vectors $\textbf{n}_{f}$ and  $\textbf{n}_{p}$ denote the unit outward normal vectors on $\partial \Omega_{f}$ and $\partial \Omega_{p}$, respectively.
And the vector $\tau_{i}(i=1,2,\cdots,d-1)$ represents the unit tangential vector on the interface $\Gamma$. Refer to Figure \ref{fig:domain} for specifics.
\begin{figure}
	\centering
	\includegraphics[width=0.6\linewidth]{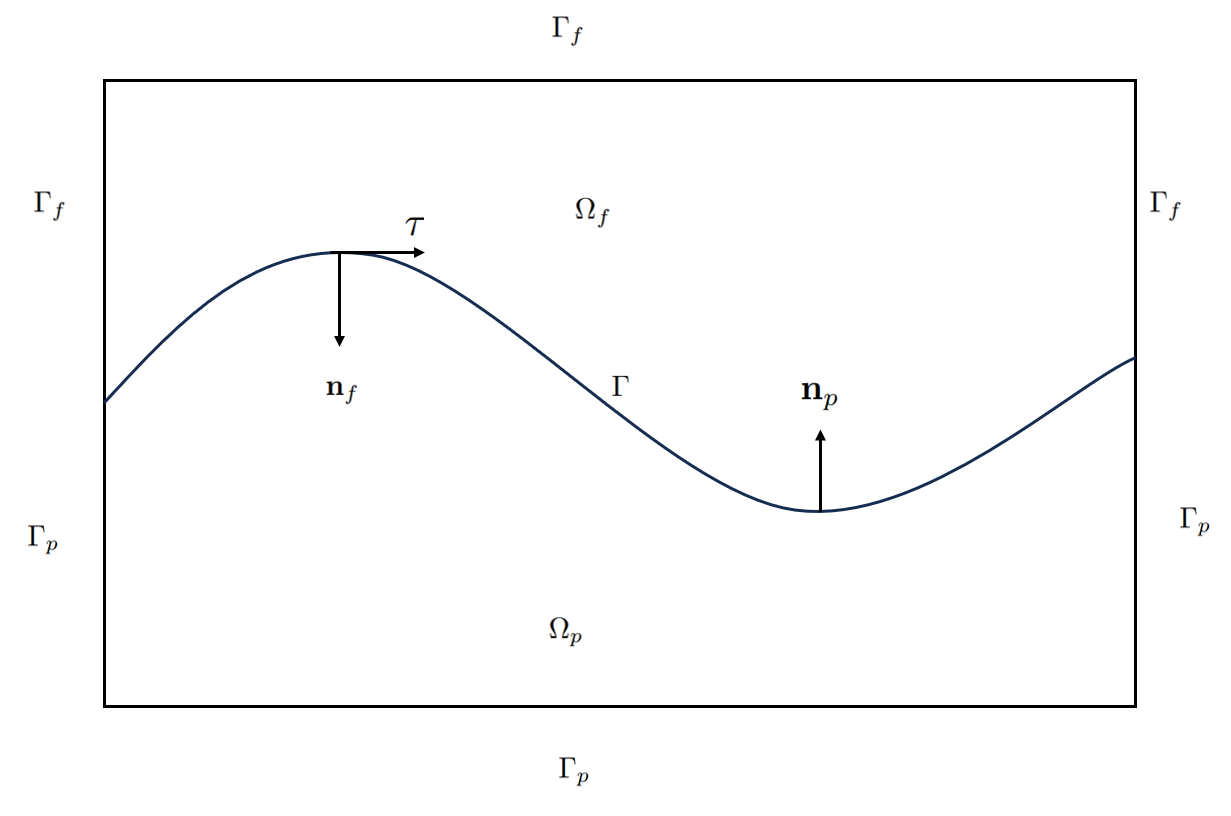}
	\caption{A global domain $\Omega$ consisting of a fluid flow region $\Omega_f$   and a porous media flow region $\Omega_p$ separated by an interface $\Gamma$.}
	\label{fig:domain}
\end{figure}
Let $T\textgreater 0$ be a finite time. The fluid flow in $\Omega_{f}$ is governed by the Stokes equation:
\begin{align}  
   \partial_{t}\textbf{u}_{f}   - \nabla \cdot \mathbb{T}(\textbf{u}_{f} , p_{f}) &= \textbf{f}_{f} \quad& \mbox{in }\Omega_{f} \times (0,T],\\
  -\nabla  \cdot \textbf{u}_{f} &= 0 \quad &\mbox{in }\Omega_{f} \times (0,T],\\
  \textbf{u}_{f}(\textbf{x},0)&=\textbf{u}^{0} \quad&\mbox{in }\Omega_{f},\\
  \textbf{u}_{f}&=0 \quad&\mbox{on } \Gamma_{f} \times (0,T], 
\end{align}
where $\mathbb{T}(\textbf{u}_{f} , p_{f})=-p_{f}I+2\nu \mathbb{D}(\textbf{u}_{f})$, $\mathbb{D}(\textbf{u}_{f})=\frac{1}{2}(\nabla \textbf{u}_{f} + \nabla^{T} \textbf{u}_{f})$
are the stress tensor and the deformation tensor.
 $\textbf{u}_{f}$ represents the velocity of the free fluid flow, $p_{f}$ the kinetic pressure, $\nu\textgreater0$ the kinematic viscosity and $\textbf{f}_{f}$ is external force.

In the porous media region $\Omega_{p}$, the porous media flow is controlled by the Darcy equation for the piezometric head $\phi_{p}$:
\begin{align}  
   S_{0}\partial_{t}\phi_{p}   - \nabla \cdot(\mathbb{K}^{\epsilon}(\textbf{x})\nabla\phi_{p}) & = f_{p}&\quad\mbox{in }\Omega_{p} \times (0,T],\\
   \phi_{p}(\textbf{x},0)&=\phi^{0}&\quad\mbox{in }\Omega_{p},\\
   \phi_{p} &=0 \quad&\mbox{on }\Gamma_{p} \times (0,T].
\end{align} 
Here $S_{0}$ is the specific mass storativity coefficient and $f_{p}$ is the source term.
$\mathbb{K}^{\epsilon}(\textbf{x})=(K_{ij})$ represents the hydraulic conductivity tensor, which is independent of time. We always assume that $\mathbb{K}^{\epsilon}(\textbf{x})$ is symmetric and satisfies uniform boundedness: there exist two constants $\lambda_{max}\textgreater\lambda_{min}\textgreater 0$ such that
\begin{align}
\forall \xi \in R^{2},
\lambda_{min}|\xi|^{2}\le \mathbb{K}^{\epsilon}(\textbf{x})\xi\cdot \xi\le \lambda_{max}|\xi|^{2}.
\end{align}
Additionally, we also assume that $\mathbb{K}^{\epsilon}(\textbf{x})$ is periodic, with a period of $\epsilon$.
 
The interface conditions are key to coupling the Stokes-Darcy model, we apply the following interface conditions on $\Gamma$:
\begin{align}  
    \textbf{u}_{f} \cdot \textbf{n}_{f} + \textbf{u}_{p} \cdot \textbf{n}_{p}&= 0& &\mbox{on }\Gamma \times (0,T],\label{BJS-1}\\
  - (\mathbb{T}(\textbf{u}_{f},p_{f})\cdot\textbf{n}_{f})\cdot\textbf{n}_{f}&=g(\phi_{p}-z)& &\mbox{on }\Gamma \times (0,T],\label{BJS-2}\\
  - (\mathbb{T}(\textbf{u}_{f},p_{f})\cdot\textbf{n}_{f})\cdot\textbf{$\tau$}_{i}&= \frac{\nu\alpha \sqrt{d}}{\sqrt{trace(\prod)}} \textbf{$\tau$}_{i} \cdot \textbf{u}_{f}
& & i=1,...,d-1 \quad \mbox{on }\Gamma \times (0,T],\label{BJS-3}
\end{align} 
where $\alpha$ is a non-negative constant, $g$ the gravitational acceleration, $\nu$ the kinematic viscosity of the fluid,
and $\prod$ the intrinsic permeability that satisfies the relation $\prod=\frac{\mathbb{K}^{\epsilon}(\textbf{x})\nu}{g}$.
The first condition (\ref{BJS-1}) guarantees the mass conservation across the interface $\Gamma$. The second one (\ref{BJS-2}) represents the balance of the normal force.
The last equation (\ref{BJS-3}) is the Beavers–Joseph–Saffman interface condition. It is a slip condition of velocity.

\subsection{Weak formulation of Stokes-Darcy model}
First of all, we define some Hilbert spaces:
\begin{align*} 
\textbf{X}_{f}=& \{ \textbf{v} \in (H^{1}(\Omega_{f}))^{d}: \textbf{v}=0 \enspace  on  \enspace \Gamma_{f}\},\\
X_{p}=& \{ \psi \in H^{1}(\Omega_{p}): \psi=0 \enspace  on \enspace  \Gamma_{p} \}, \\
 Q_{f}=& L^{2}_{0}(\Omega_{f})= \{ q \in L^{2}(\Omega_{f}): 
 \int_{\Omega_{f}} q=0 \},
\end{align*} 
and the product Hilbert spaces: 
\begin{align*}
\textbf{L}^{2}=(L^{2}(\Omega_{f}))^{d} \times L^{2}(\Omega_{p}), \ 
\textbf{U}=\textbf{X}_{f} \times X_{p}.
\end{align*}
From now on, we use $(\cdot,\cdot)_{D}$ and $\|\cdot\|_{D}$ to denote the $L^{2}$
inner product and the corresponding norm on any given
domain $D$. Sometimes, for simplicity and without loss of context, we also denote it as $(\cdot,\cdot)$.

The space $Q_{f}$ is equipped with the following norm:
$$\|q\|_{0}:=\|q\|_{Q_{f}}=\|q\|_{L^{2}(\Omega_{p})}, \forall q\in Q_{f}.$$
Let us denote $\underline{\textbf{v}}=(\textbf{v},\psi) \in \textbf{U}$. We equip the space $\textbf{U}$ with the following norm:
$$\|\underline{\textbf{v}}\|_{0}:=(\|\textbf{v}\|_{(L^{2}(\Omega_{f}))^{d}}^{2}+\|\psi\|_{L^{2}(\Omega_{p})}^{2}),\forall\underline{\textbf{v}}=(\textbf{v},\psi) \in \textbf{U}.$$
Particularly, we use the following notations to represent the norms:
$$\|\textbf{v}\|_{0}:=\|\textbf{v}\|_{L^{2}(\Omega_f)}, \|\nabla \textbf{v} \|_{0}= \| \textbf{v} \|_{1} := \|\textbf{v}\|_{\textbf{X}_{f}(\Omega_{f})}, $$
 $$\|\psi\|_{0}:=\|\psi\|_{L^{2}(\Omega_p)},
 \|\nabla \psi \|_{0}= \| \psi \|_{1} := \|\psi\|_{X_{p}(\Omega_{p})}. $$
Assume that $\textbf{f}_{f} \in L^{2}(0,T;L^{2}(\Omega_{f})^{d})$, $f_{p}\in L^{2}(0,T;L^{2}(\Omega_{p}))$. Then the weak formulation for the non-stationary Stokes-Darcy problem is : 
find $\underline{\textbf{u}}=(\textbf{u},\phi)\in ((L^{2}(0,T;\textbf{X}_{f})\cap L^{\infty}(0,T;L^{2}(\Omega_{f})^{d})) \times ((L^{2}(0,T;X_{p})\cap L^{\infty}(0,T;L^{2}(\Omega_{p})))$ and
$p\in L^{2}(0,T;Q_{f})$, such that 
\begin{align} 
   (\partial_{t} \textbf{\underline{u}},\textbf{\underline{v}}) +a(\textbf{\underline{u}},\textbf{\underline{v}}) 
   +b(\textbf{\underline{v}},p)
   & =(\textbf{F},\textbf{\underline{v}})  \quad 
 \forall \textbf{\underline{v}} \in \textbf{U},\notag \\
   b(\textbf{\underline{u}},q)& =0
   \quad \qquad \forall q \in Q_{f},\label{wf} \\ 
\textbf{\underline{u}}(0) &=\textbf{\underline{u}}^{0}, \notag
\end{align}
where
\begin{align*}
     (\partial_{t} \textbf{\underline{u}},\textbf{v}) &=(\partial_{t}\textbf{u},\textbf{v})_{\Omega_f}  
     +(\partial_{t} \phi,\psi)_{\Omega_p} ,                        \\
      a(\textbf{\underline{u}},\textbf{\underline{v}}) &= a_{f}(\textbf{u},\textbf{v})+a_{p}(\phi,\psi)
      +a_{\Gamma}(\textbf{\underline{u}}, \textbf{\underline{v}}) ,        \\
     a_{f}(\textbf{u},\textbf{v}) &= 
     2\nu (\mathbb{D}(\textbf{u}),\mathbb{D}(\textbf{v}))_{\Omega_f}
     + \int_{\Gamma} \frac{\nu \alpha \sqrt{d}}{\sqrt{trace(\prod)}} (\textbf{u} \cdot \tau_{i})(\textbf{v} \cdot \tau_{i}),      \\
     a_{p}(\phi,\psi) &= \frac{1}{S_{0}}(\mathbb{K}^{\epsilon}(\textbf{x}) \nabla \phi ,\nabla \psi)_{\Omega_p} ,\\
 a_{\Gamma}(\textbf{\underline{u}},\textbf{\underline{v}}) &= 
     g \int_{\Gamma} \phi\textbf{v} \cdot \textbf{n}_{f} 
     -\frac{1}{S_{0}} \int_{\Gamma}\textbf{u}\cdot \textbf{n}_{f} \psi,                                                          \\
     b(\textbf{\underline{v}},p)&= -(p,\nabla \cdot \textbf{v})_{\Omega_{f}} ,                                                      \\
     (\textbf{F},\textbf{\underline{v}}) &= (\textbf{f}_{f},\textbf{v})_{\Omega_f} + \frac{1}{S_{0}} (f_{p},\psi)_{\Omega_p} . \\  
\end{align*}
\begin{remark}
The well-posedness of the coupled problem (\ref{wf}) is given in previous literatures \cite{DMQ,LSY} and will not be explained in this paper.
\end{remark}

\subsection{Time and space discretization}
We construct the regular triangulation $\mathcal{T}_{fh}$, $\mathcal{T}_{ph}$ of $\Omega_{f}$ and $\Omega_{p}$, where $\overline{\Omega}_{p}$=$\cup_{K \in \{{\mathcal{T}_{ph}}\} } K $. 
At the same time, we uniformly divide the time interval $[0, T]$ into $N$ subintervals $[t^{n}, t^{n+1}]$ for $n = 0, 1, \ldots, N-1$, such that $0 = t^{0} < t^{1} < \cdots < t^{N} = T$. Here $\Delta t = t^{n+1} - t^{n}$ represents the time step length.

Let $\textbf{X}_{fh} \subset \textbf{X}_{f}$, and $Q_{fh} \subset Q_{f}$ be finite element spaces such that the space pair $(\textbf{X}_{fh},Q_{fh})$ satisfies the discrete LBB condition: there exists a constant $\beta >0$, independent of mesh size, such that
\begin{align}\label{LBB}
\underset{0 \neq q_{h}\in Q_{fh}}{inf}  \underset{0 \neq \textbf{v}_{h}\in \textbf{X}_{fh}}{sup}  
\frac{(q_{h},\nabla \cdot \textbf{v}_{h})_{\Omega_{f}}}{\|\textbf{v}_{h}\|_{1}  \|q_{h}\|_{0}}> \beta.
\end{align}
We denote $X_{ph}\subset X_{p}$ and $\textbf{U}_{h}=\textbf{X}_{fh}\times X_{ph}$.
Here we  choose MINI finite element pair for $(\textbf{X}_{fh}, Q_{fh})$ and multiscale finite element for $X_{ph}$. Next we introduce the multiscale finite element basis functions for the Darcy region.

In this paper, although the Darcy equation is a parabolic equation, the permeability coefficient we consider is a periodic function in space, which is independent of time.
 So the approach to constructing the multiscale basis functions can be aligned with that used for elliptic equation. In reference \cite{HWC}, Hou constructed multiscale basis functions that primarily used for addressing two-dimensional second-order elliptic boundary value problems with high oscillation coefficients. Drawing inspiration from his method, we next present the steps for constructing multiscale basis functions for the Darcy region in this paper.

 First, by removing the time derivative term from the Darcy equation, it is transformed into an elliptic equation. Next, in each triangulation element $K$, define a set of node basis functions  $\{\eta_{K}^{i},i=1,2,3\}$. To ensure that the basis functions incorporate the microscale information of the permeability coefficients in local elements, we set $\eta_{K}^{i}$ to satisfy the following local homogeneous sub-problem:
\begin{align}\label{basis}
 - \nabla \cdot(\mathbb{K}^{\epsilon}(\textbf{x})\nabla\eta_{K}^{i})  =0 \quad \mbox{in } K \in \mathcal{T}_{ph}.
\end{align}
Let $\textbf{x}_j$ $(j=1,2,3)$ denote the nodal points of $K$. Here define $\eta_{K}^{i}(\textbf{x}_j)=\delta_{ij}$, indicating that $\eta_{K}^{i}(\textbf{x}_i)=1$ when $i=j$ and $\eta_{K}^{i}(\textbf{x}_j)=0$ when $i\neq j$. To ensure the well-posedness of problem (\ref{basis}), specifying boundary conditions for $\eta_{K}^{i}$ on the boundary of $K$ becomes necessary.
Thus, in this paper, we consider the multiscale basis function $\eta_{K}^{i}=\theta_{K}^{i}$ on $\partial K$, where $\theta_{K}^{i}$ is a linear function defined on $\partial K$.
In conclusion, $\eta_{K}^{i}$ satisfies the following equation:
\begin{align}\label{Kbasis}
\begin{cases}
    - \nabla \cdot(\mathbb{K}^{\epsilon}(\textbf{x})\nabla\eta_{K}^{i})  =0 \quad &\mbox{in } \quad K,\\
   \eta_{K}^{i}=\theta_{K}^{i}, & on \quad\partial K,\\
   \eta_{K}^{i}(x_{j})=\delta_{ij} .\\
\end{cases}
\end{align}

Finally, we apply the standard finite element method to solve problem (\ref{Kbasis}) on the local element $K$.
The solution region $K$ is partitioned into a finer mesh $\mathcal{T}_{ph_{fine}}^K$ with a mesh size of $h_{fine}$. Subsequently, the P1 element is utilized to solve the equations (\ref{Kbasis}) \cite{HWC,HTW}.

Hence, we can obtain a space consisting of multiscale basis functions in the Darcy region: $$X_{ph}=span\{\eta_{K}^{i},i=1,2,3,K \in \mathcal{T}_{ph} \}.$$

Moreover, we provide the trace, $\text {Poincaré}$ and Korn's inequalities: there exist constants $C_{t}$,$C_{p}$,$C_{1}$ that only depend on the region $\Omega_{f}$, and 
$\widetilde{C}_{t}$, $\widetilde{C}_{p}$ that only depend on the region $\Omega_{p}$,     
 such that for all $\textbf{v}_{h} \in \textbf{X}_{fh}$ and $\psi_{h}\in X_{ph}$,
\begin{align*}
&\|\textbf{v}_{h}\|_{L^{2}(\Gamma)} \leq C_{t}\|\textbf{v}_{h}\|_{0}^{\frac{1}{2}}\|\textbf{v}_{h}\|_{1}^{\frac{1}{2}},&
&\|\psi_{h}\|_{L^{2}(\Gamma)} \leq\widetilde{C}_{t}\|\psi_{h}\|_{0}^{\frac{1}{2}}\|\psi_{h}\|_{1}^{\frac{1}{2}},\\
&\|\textbf{v}_{h}\|_{0}\leq C_{p}\|\textbf{v}_{h}\|_{1},&
&\|\psi_{h}\|_{0}\leq \widetilde{C}_{p}\|\psi_{h}\|_{1},\\
&(\mathbb{D}(\textbf{v}_{h}),\mathbb{D}(\textbf{v}_{h}))\geq C_{1}\|\textbf{v}_{h}\|_{1}^{2}.&
\end{align*}

\section{The multiscale finite element method for the Stokes-Darcy model}
\subsection{Numerical algorithm}
In this part, we present the algorithm for solving the non-stationary Stokes-Darcy model, where the Darcy region exhibits multiscale phenomena. We divide it into two processes: offline and online. 
Before formally presenting the algorithm, we first compute the multiscale basis functions for the Darcy region(Pre-work).
\floatname{algorithm}{}   
\begin{algorithm}[H]
\renewcommand*{\thealgorithm}{} 
	\caption*{\textbf{Pre-work}: MsFEM basis in Darcy region (offline)} 
	\label{Pre-work} 
	\begin{algorithmic}
	      \Require Coarse mesh $\mathcal{T}_{ph}$, the hydraulic conductivity tensor $\mathbb{K}^{\epsilon}(\textbf{x})$ and a small mesh size $h_{fine}$(we need $h_{fine} \ll$ $\epsilon$).
		\State Initialize the stiffness matrix $A_{1}$, mass matrix $A_{2}$ and the vector $B$ of the multiscale basis.
		\label{ code:fram:extract }
		\For{$K\in\mathcal{T}_{ph}$ in parallel}
		\State Build $\mathcal{T}_{ph_{fine}}^K$ (meshing $K$ with mesh size $h_{fine}$)
		\For{$i \in K$ in parallel} 
		\State Solve with P1 FEM 
             $$-\nabla\cdot(\mathbb{K}^{\epsilon}(\textbf{x})\nabla \eta_K^{i})=0 \quad in \  K, 
		\quad \eta_K^{i}=\theta_K^{i}\quad on\ \partial K,$$
		\State where $\theta_K^{i}$ is the standard P1 FEM basis function, and store $\eta_K^{i}$.
		\EndFor	
		\State Compute and store $(A_{1})_{i,j}^{K}=\int_{K}(\mathbb{K}^{\epsilon}(\textbf{x})\nabla \eta_K^i)(\nabla\eta_K^j),$ $ (A_{2})_{i,j}^{K}=\int_{K} \eta_K^i \eta_K^j$ 
\State and $B_i^K=B_i^K+\int_{K}\eta_K^i f_p$
		with $A_{1}^{K}$, $A_{2}^{K}$ the local matrix and the local $B^K$ vector. 
		\EndFor
		\Ensure Assemble and store $A_{1}$, $A_{2}$ and $B$ associated with the multiscale basis.
	\end{algorithmic}
\end{algorithm}

Next, we present an implicit-explicit(ImEx) scheme based on the multiscale finite element method for the Stokes-Darcy model.
\begin{algorithm}[H] 
     \caption{}
     \caption*{\textbf{Algorithm 1} MsFEM-ImEx}
	\label{alg:1} 
	\begin{algorithmic}
       \State $\bullet$ \textbf{offline:} Pre-work 
       \State $\bullet$ \textbf{online:} solve the Stokes-Darcy problem
        \For{$n=0,1\cdots N-1 $}
       \State Find $(\textbf{u}_{h}^{n+1},p_{h}^{n+1})\in \textbf{X}_{fh}\times Q_{fh}$ such that $\forall (\textbf{v}_{h},q_{h})\in \textbf{X}_{fh}\times Q_{fh} $,
       \begin{align}
        (\frac{\textbf{u}_{h}^{n+1}-\textbf{u}_{h}^{n}}{\Delta t},\textbf{v}_{h})+a_{f}(\textbf{u}_{h}^{n+1},\textbf{v}_{h})+b(\textbf{v}_{h},p_{h}^{n+1})&=(\textbf{f}_{f}^{n+1},\textbf{v}_{h})-g\int_{\Gamma}\phi_{h}^{n}\textbf{v}_{h}\cdot \textbf{n}_{f},\label{Sa1} \\  
b(\textbf{u}_{h}^{n+1},q_{h})&=0.  \label{Sa2} 
\end{align}
        \State Find $\phi_{h}^{n+1}\in X_{ph}$ such that $\forall \psi_{h}\in X_{ph}$,
     \begin{align}
     (\frac{\phi_{h}^{n+1}-\phi_{h}^{n}}{\Delta t},\psi_{h})+a_{p}(\phi_{h}^{n+1},\psi_{h})=\frac{1}{S_{0}}(f_{p}^{n+1},\psi_{h})+\frac{1}{S_{0}}\int_{\Gamma} \textbf{u}_{h}^{n}\cdot \textbf{n}_{f} \psi_{h}. \label{Da}
     \end{align}
       \EndFor

     \end{algorithmic}
\end{algorithm}

\subsection{Stability analysis}
In this part, we prove the stability result of Algorithm \ref{alg:1} (MsFEM-ImEx).
\begin{theorem}[Stability]\label{Stability}
Under the assumption $$\widetilde{C} \Delta t \textless 1, \ \widetilde{C}=\frac{C_{t}^{2}\widetilde{C}_{t}^{2}}{2\sqrt{S_{0}^{3}C_{1}\nu \lambda_{min}}}+\frac{g^{2}C_{t}^{2}\widetilde{C}_{t}^{2}\sqrt{S_{0}}}{\sqrt{2C_{1}\nu \lambda_{min}}},$$
we have the following stability result:
\begin{equation}\label{stability_result}
{
\begin{aligned}
&\|\textbf{u}_{h}^{N}\|_{0}^{2} + \|\phi_{h}^{N}\|_{0}^{2}
 +\sum\limits_{n=0}^{N-1}(\|\textbf{u}_{h}^{n+1}-\textbf{u}_{h}^{n}\|_{0}^{2} +\|\phi_{h}^{n+1}-\phi_{h}^{n}\|_{0}^{2}) 
 + \frac{C_{1}\nu\Delta t}{2}\|\textbf{u}_{h}^{N}\| _{1}^{2}
 + \frac{\lambda_{min}\Delta t}{2S_{0}}\| \phi_{h}^{N}\| _{1}^{2}   \\
 &\qquad \quad  \le C(T) (
\frac{C_{p}^{2}}{3C_{1}\nu}\Delta t \sum\limits_{n=0}^{N-1}\|\textbf{f}_{f}^{n+1}\|_{0}^{2}   
+\frac{\widetilde{C}_{p}^{2}}{S_{0}\lambda_{min}}\Delta t \sum\limits_{n=0}^{N-1} \|f_{p}^{n+1}\|_{0}^{2} \\
&\qquad \qquad +\|\textbf{u}^{0}\|_{0}^{2} + \|\phi^{0}\|_{0}^{2}
+ \frac{C_{1}\nu\Delta t}{2}\|\textbf{u}^{0}\| _{1}^{2}
 + \frac{\lambda_{min}\Delta t}{2S_{0}}\| \phi^{0}\|_{1}^{2}), 
\end{aligned}
}
\end{equation}
where $C(T)$ is a positive constant which depends on the final time $T$. The constants $C_{1}, C_{t}, \widetilde{C}_{t}, {C}_{p}$ and $\widetilde{C}_{p}$ are respectively associated with Korn's inequality,
the trace inequality, and the $\text {Poincaré}$ inequality.
\end{theorem}
\begin{proof}
Set $\textbf{v}_{h}=2\Delta t\textbf{u}_{h}^{n+1}$ in (\ref{Sa1}). Then using the equality $2(a-b,a)=\vert a\vert ^{2} - \vert b\vert ^{2} +\vert a-b\vert ^{2}$ and  $b(\textbf{u}_{h}^{n+1},p_{h}^{n+1})=0$, we obtain
\begin{align}\label{S1}
\|\textbf{u}_{h}^{n+1}\|_{0}^{2} 
-\|\textbf{u}_{h}^{n}\|_{0}^{2}
 &+\|\textbf{u}_{h}^{n+1}-\textbf{u}_{h}^{n}\|_{0}^{2} 
 +4\nu \Delta t  \|\mathbb{D}(\textbf{u}_{h}^{n+1}) \|_{0}^{2}
 +\Delta t \int_{\Gamma} \frac{2 \nu \alpha \sqrt{d}}{\sqrt{trace(\prod)}}(\textbf{u}_{h}^{n+1} \cdot \tau_{i})(\textbf{u}_{h}^{n+1} \cdot \tau_{i}),      \notag  \\       
&=2\Delta t (\textbf{f}_{f}^{n+1},\textbf{u}_{h}^{n+1})_{\Omega_{f}}
-2\Delta tg \int_{\Gamma}\phi_{h}^{n}\textbf{u}_{h}^{n+1} \cdot \textbf{n}_{f}. 
\end{align}
Setting $\psi_{h}=2\Delta t \phi_{h}^{n+1}$and using $2(a-b,a)=\vert a\vert ^{2} - \vert b\vert ^{2} +\vert a-b\vert ^{2}$, we have
\begin{align}\label{S2}
\|\phi_{h}^{n+1}\|_{0}^{2} 
-\|\phi_{h}^{n}\|_{0}^{2}
 +&\|\phi_{h}^{n+1}-\phi_{h}^{n}\|_{0}^{2} 
 +\frac{2\Delta t}{S_{0}}(\mathbb{K}^{\epsilon}(\textbf{x})\nabla \phi_{h}^{n+1},\nabla \phi_{h}^{n+1})_{\Omega_{p}} \notag \\
 =&\frac{2\Delta t}{S_{0}}(f_{p}^{n+1},\phi_{h}^{n+1})_{\Omega_{p}}
 + \frac{2\Delta t}{S_{0}}\int_{\Gamma}\textbf{u}_{h}^{n} \cdot \textbf{n}_{f}\phi_{h}^{n+1}.
\end{align}
Combining (\ref{S1}) with (\ref{S2}) and using Korn's inequality as well as the uniform boundedness of permeability tensor, we get
\begin{align}\label{S3}
&\|\textbf{u}_{h}^{n+1}\|_{0}^{2} 
-\|\textbf{u}_{h}^{n}\|_{0}^{2}
 +\|\textbf{u}_{h}^{n+1}-\textbf{u}_{h}^{n}\|_{0}^{2} 
 +\|\phi_{h}^{n+1}\|_{0}^{2} 
-\|\phi_{h}^{n}\|_{0}^{2}
 +\|\phi_{h}^{n+1}-\phi_{h}^{n}\|_{0}^{2}       \nonumber \\
 &+4C_{1}\nu \Delta t  \ \|\textbf{u}_{h}^{n+1} \|_{1}^{2}  
 +\frac{2\lambda_{min}}{S_{0}}\Delta t\ \|\phi_{h}^{n+1}\|_{1}^{2}  \nonumber  \\
 \le &2\Delta t (\textbf{f}_{f}^{n+1},\textbf{u}_{h}^{n+1})_{\Omega_{f}}
 +\frac{2\Delta t}{S_{0}}(f_{p}^{n+1},\phi_{h}^{n+1})_{\Omega_{p}}  \nonumber  \\
 &+\frac{2\Delta t}{S_{0}}\int_{\Gamma}\textbf{u}_{h}^{n} \cdot \textbf{n}_{f}\phi_{h}^{n+1}
  -2\Delta tg \int_{\Gamma}\phi_{h}^{n}\textbf{u}_{h}^{n+1} \cdot \textbf{n}_{f}.
\end{align}
For the first two terms on the right-hand side of (\ref{S3}), by using $\text {Hölder}$, $\text {Poincaré}$ and Young inequalities, it follows that
\begin{equation}\label{S4}
{
\begin{aligned}
2\Delta t (\textbf{f}_{f}^{n+1},\textbf{u}_{h}^{n+1})_{\Omega_{f}}
 +\frac{2\Delta t}{S_{0}}(f_{p}^{n+1},\phi_{h}^{n+1})_{\Omega_{p}}
 &\le 3C_{1}\nu \Delta t \| \textbf{u}_{h}^{n+1}\|_{1}^{2}
 +\frac{C_{p}^{2}}{3C_{1}\nu}\Delta t \|\textbf{f}_{f}^{n+1}\|_{0}^{2}  \\
 &+\frac{\lambda_{min}}{S_{0}}\Delta t \| \phi_{h}^{n+1}\|_{1}^{2}
 +\frac{\widetilde{C}_{p}^{2}}{S_{0}\lambda_{min}}\Delta t \|f_{p}^{n+1}\|_{0}^{2}.   
\end{aligned}
}
\end{equation}
The third and fourth terms on the right-hand side of (\ref{S3}) are bounded by $\text {Hölder}$, trace and Young inequalities, we can obtain
\begin{align}\label{S5}
&\frac{2\Delta t}{S_{0}}\int_{\Gamma}\textbf{u}_{h}^{n} \cdot \textbf{n}_{f}\phi_{h}^{n+1}
  -2\Delta tg \int_{\Gamma}\phi_{h}^{n}\textbf{u}_{h}^{n+1} \cdot \textbf{n}_{f} \nonumber \\
&\le \frac{2\Delta t}{S_0}\|\textbf{u}_{h}^{n}\|_{L^{2}(\Gamma)}\|\phi_{h}^{n+1}\|_{L^{2}(\Gamma)}
+2g\Delta t \|\phi_{h}^{n}\|_{L^{2}(\Gamma)} \|\textbf{u}_{h}^{n+1}\|_{L^{2}(\Gamma)} \nonumber \\
&\le \frac{2C_{t}\widetilde{C}_{t}\Delta t}{S_0}\|\textbf{u}_{h}^{n}\|_{0}^{\frac{1}{2}} 
\|   \textbf{u}_{h}^{n}\|_{1}^{\frac{1}{2}}
\|\phi_{h}^{n+1}\|_{0}^{\frac{1}{2}} 
\| \phi_{h}^{n+1}\|_{1}^{\frac{1}{2}}
+2gC_{t}\widetilde{C}_{t}\Delta t \|\phi_{h}^{n}\|_{0}^{\frac{1}{2}} \| \phi_{h}^{n}\|_{1}^{\frac{1}{2}} \|\textbf{u}_{h}^{n+1}\|_{0}^{\frac{1}{2}} \|   \textbf{u}_{h}^{n+1}\|_{1}^{\frac{1}{2}} \nonumber  \\
&\le \sqrt{\frac{C_{1}\nu \lambda_{min}}{S_{0}}}\Delta t 
\|   \textbf{u}_{h}^{n}\|_{1} 
\| \phi_{h}^{n+1}\|_{1}
+\frac{C_{t}^{2}\widetilde{C}_{t}^{2}}{\sqrt{S_{0}^{3}C_{1}\nu \lambda_{min}}}\Delta t
\|\textbf{u}_{h}^{n}\|_{0} 
\|\phi_{h}^{n+1}\|_{0}                    \nonumber  \\
&\qquad +\sqrt{\frac{C_{1}\nu \lambda_{min}}{2S_{0}}}\Delta t 
\|   \textbf{u}_{h}^{n+1}\|_{1} 
\| \phi_{h}^{n}\|_{1}
+\sqrt{\frac{2S_{0}}{C_{1}\nu \lambda_{min}}}g^{2}C_{t}^{2}\widetilde{C}_{t}^{2} \Delta t
\|\textbf{u}_{h}^{n+1}\|_{0} 
\|\phi_{h}^{n}\|_{0}       \nonumber \\
&\le \frac{C_{1}\nu}{2}\Delta t \|   \textbf{u}_{h}^{n}\|_{1}^{2}
+\frac{\lambda_{min}}{2S_{0}}\Delta t   \| \phi_{h}^{n+1}\|_{1}^{2}
+\frac{C_{1}\nu}{2}\Delta t \|   \textbf{u}_{h}^{n+1}\|_{1}^{2}
+\frac{\lambda_{min}}{4S_{0}}\Delta t   \| \phi_{h}^{n}\|_{1}^{2}  \nonumber   \\
&\qquad +\frac{C_{t}^{2}\widetilde{C}_{t}^{2} }{2\sqrt{S_{0}^{3}C_{1}\nu \lambda_{min}}}\Delta t (\|\textbf{u}_{h}^{n}\|_{0}^{2} + \|\phi_{h}^{n+1}\|_{0}^{2})
+\frac{g^{2}C_{t}^{2}\widetilde{C}_{t}^{2} \sqrt{S_0}}{\sqrt{2C_{1}\nu \lambda_{min}}} \Delta t (\|\textbf{u}_{h}^{n+1}\|_{0}^{2} + \|\phi_{h}^{n}\|_{0}^{2}).
\end{align}
Combining (\ref{S4}) and (\ref{S5}) with (\ref{S3}) leads to
\begin{align}\label{S6}
&\|\textbf{u}_{h}^{n+1}\|_{0}^{2} 
-\|\textbf{u}_{h}^{n}\|_{0}^{2}
 +\|\textbf{u}_{h}^{n+1}-\textbf{u}_{h}^{n}\|_{0}^{2} 
 +\|\phi_{h}^{n+1}\|_{0}^{2} 
-\|\phi_{h}^{n}\|_{0}^{2}
+\|\phi_{h}^{n+1}-\phi_{h}^{n}\|_{0}^{2}  \nonumber \\
&\quad +\frac{C_{1}\nu }{2}\Delta t  (\| \textbf{u}_{h}^{n+1}\|_{1}^{2} -\| \textbf{u}_{h}^{n} \|_{1}^{2} )
+\frac{\lambda_{min}}{2S_{0}}\Delta t\| \phi_{h}^{n+1}\|_{1}^{2} -\frac{\lambda_{min}}{4S_{0}}\Delta t \| \phi_{h}^{n}\|_{1}^{2}  \nonumber \\
&\le \frac{C_{p}^{2}}{3C_{1}\nu}\Delta t \|\textbf{f}_{f}^{n+1}\|_{0}^{2}   
+\frac{\widetilde{C}_{p}^{2}}{S_{0}\lambda_{min}}\Delta t \|f_{p}^{n+1}\|_{0}^{2}   \nonumber  \\
&\quad +\frac{C_{t}^{2}\widetilde{C}_{t}^{2}}{2\sqrt{S_{0}^{3}C_{1}\nu \lambda_{min}}}\Delta t (\|\textbf{u}_{h}^{n}\|_{0}^{2} + \|\phi_{h}^{n+1}\|_{0}^{2})
+\frac{g^{2}C_{t}^{2}\widetilde{C}_{t}^{2}\sqrt{S_{0}}}{\sqrt{2C_{1}\nu \lambda_{min}}} \Delta t (\|\textbf{u}_{h}^{n+1}\|_{0}^{2} + \|\phi_{h}^{n}\|_{0}^{2}).  
\end{align}
Let $\widetilde{C}=\frac{C_{t}^{2}\widetilde{C}_{t}^{2}}{2\sqrt{S_{0}^{3}C_{1}\nu \lambda_{min}}}+\frac{g^{2}C_{t}^{2}\widetilde{C}_{t}^{2}\sqrt{S_{0}}}{\sqrt{2C_{1}\nu \lambda_{min}}}$. Here $\frac{\lambda_{min}}{2S_{0}}>\frac{\lambda_{min}}{4S_{0}}$ is evident.
Then by summing up equation (\ref{S6}) from $n=0$ to $N-1$, we have
\begin{align}\label{S7}
&\|\textbf{u}_{h}^{N}\|_{0}^{2} + \|\phi_{h}^{N}\|_{0}^{2}
 +\sum\limits_{n=0}^{N-1}(\|\textbf{u}_{h}^{n+1}-\textbf{u}_{h}^{n}\|_{0}^{2} +\|\phi_{h}^{n+1}-\phi_{h}^{n}\|_{0}^{2}) 
 + \frac{C_{1}\nu\Delta t}{2}\|\textbf{u}_{h}^{N}\| _{1}^{2}
 + \frac{\lambda_{min}\Delta t}{2S_{0}}\| \phi_{h}^{N}\| _{1}^{2}   \nonumber \\
 &\qquad \quad  \le \widetilde{C} \Delta t \sum\limits_{n=0}^{N-1}(\|\textbf{u}_{h}^{n+1}\|_{0}^{2} + \|\phi_{h}^{n+1}\|_{0}^{2})
 +\frac{C_{p}^{2}}{3C_{1}\nu}\Delta t \sum\limits_{n=0}^{N-1}\|\textbf{f}_{f}^{n+1}\|_{0}^{2}   
+\frac{\widetilde{C}_{p}^{2}}{S_{0}\lambda_{min}}\Delta t \sum\limits_{n=0}^{N-1} \|f_{p}^{n+1}\|_{0}^{2} \nonumber \\
&\qquad \qquad +\|\textbf{u}^{0}\|_{0}^{2} + \|\phi^{0}\|_{0}^{2}
+ \frac{C_{1}\nu\Delta t}{2}\|\textbf{u}^{0}\| _{1}^{2}
 + \frac{\lambda_{min}\Delta t}{2S_{0}}\| \phi^{0}\|_{1}^{2}. 
\end{align}
When $\widetilde{C} \Delta t \textless 1$, employing Gronwall inequality allows us to obtain the following result
\begin{align}\label{S8}
&\|\textbf{u}_{h}^{N}\|_{0}^{2} + \|\phi_{h}^{N}\|_{0}^{2}
 +\sum\limits_{n=0}^{N-1}(\|\textbf{u}_{h}^{n+1}-\textbf{u}_{h}^{n}\|_{0}^{2} +\|\phi_{h}^{n+1}-\phi_{h}^{n}\|_{0}^{2}) 
 + \frac{C_{1}\nu\Delta t}{2}\|\textbf{u}_{h}^{N}\| _{1}^{2}
 + \frac{\lambda_{min}\Delta t}{2S_{0}}\| \phi_{h}^{N}\| _{1}^{2} \nonumber  \\
 &\qquad \quad  \le C(T) (
\frac{C_{p}^{2}}{3C_{1}\nu}\Delta t \sum\limits_{n=0}^{N-1}\|\textbf{f}_{f}^{n+1}\|_{0}^{2}   
+\frac{\widetilde{C}_{p}^{2}}{S_{0}\lambda_{min}}\Delta t \sum\limits_{n=0}^{N-1} \|f_{p}^{n+1}\|_{0}^{2} \nonumber\\
&\qquad \qquad +\|\textbf{u}^{0}\|_{0}^{2} + \|\phi^{0}\|_{0}^{2}
+ \frac{C_{1}\nu\Delta t}{2}\|\textbf{u}^{0}\| _{1}^{2}
 + \frac{\lambda_{min}\Delta t}{2S_{0}}\| \phi^{0}\|_{1}^{2}). 
\end{align}
\end{proof}
\subsection{Error estimates}
In this part, we will analyze the error of Algorithm \ref{alg:1}. Here we assume the true solution 
$\underline{\textbf{u}}(t)=(\textbf{u}(t),\phi(t))$ satisfies the following assumptions:
\begin{align*}
\underline{\textbf{u}}_{t}(t) \in L^{2}(0,T;\textbf{U}),\ \underline{\textbf{u}}_{tt}(t) \in L^{2}(0,T;\textbf{L}^{2}).
\end{align*}
Next we define a projection operator $P_{h}:(\underline{\textbf{u}}(t),p(t))\in \textbf{U}\times Q_{f}\rightarrow (P_{h}\underline{\textbf{u}}(t),P_{h}p(t))\in \textbf{U}_{h}\times Q_{fh},$
$\forall t\in [0,T]$  by satisfying
\begin{align*} 
a(P_{h}\underline{\textbf{u}}(t),\underline{\textbf{v}}_{h})
+b(\underline{\textbf{v}}_{h},P_{h}p(t))
&=a(\underline{\textbf{u}}(t),\underline{\textbf{v}}_{h})
+b(\underline{\textbf{v}}_{h},p(t)) &\forall \underline{\textbf{v}}_{h}\in \textbf{U}_{h},\\
b(P_{h}\underline{\textbf{u}}(t),q_{h})&=0 &\forall q_{h}\in Q_{fh}.
\end{align*}
In addition, for any $t \in [0,T]$, we assume that the solution satisfies $\textbf{\underline{u}}(t) \in H^{2}(\Omega_{f})^d \times H^{2}(\Omega_{p})$ and $p(t)\in H^{1}(\Omega_{f})$.
If $\mathbb{K}^{\epsilon}(\textbf{x})$ doesn't exhibit multiscale characteristics, the following approximate properties hold:
\begin{align}  
\|P_{h}\underline{\textbf{u}}(t)-\underline{\textbf{u}}(t)\|_{0}&\le Ch^{2},\\
 \|P_{h}\underline{\textbf{u}}(t)-\underline{\textbf{u}}(t)\|_{1}&\le Ch.
\end{align}
However, when $\mathbb{K}^{\epsilon}(\textbf{x})$ has multiscale characteristics, if we only consider the elliptic equation in a single domain without involving coupled problems. In \cite{YDC}, the multiscale finite element method is applied to elliptic equation, demonstrating the $L^2$ and $H^1$ error estimation.
Therefore, it is natural for us to conjecture that similar properties hold true for coupled problems:
\begin{align}   
\|P_{h}\underline{\textbf{u}}(t)-\underline{\textbf{u}}(t)\|_{0}&\le C(\epsilon+h^{2}+\frac{\epsilon}{h}),\label{P1}\\
 \|P_{h}\underline{\textbf{u}}(t)-\underline{\textbf{u}}(t)\|_{1}&\le C(\sqrt{\epsilon}+h+\sqrt{\frac{\epsilon}{h}})\label{P2}.   
\end{align}
\begin{remark}
By introducing an interpolation operator in the Darcy region and a projection operator in the Stokes region, combined with relevant conclusions from existing literature on projection and interpolation, we demonstrated the $L^2$ and $H^1$ error estimates for the steady-state Stokes-Darcy model with multiscale characteristics in the Darcy region \cite{HZZZ}. Therefore, the properties mentioned above hold.
\end{remark}
For simplicity, we use $(\widetilde{\underline{\textbf{u}}}(t),\widetilde{\phi}(t))$ replace $(P_{h}\underline{\textbf{u}}(t),P_{h}p(t))$. And we denote 
$$e_{1}^{n+1}=\textbf{u}(t^{n+1})
-\widetilde{\textbf{u}}^{n+1}
+\widetilde{\textbf{u}}^{n+1}
-\textbf{u}_{h}^{n+1}:= \sigma_{1}^{n+1} + \xi_{1}^{n+1},$$
$$e_{2}^{n+1}=\phi(t^{n+1})
-\widetilde{\phi}^{n+1}
+\widetilde{\phi}^{n+1}
-\phi_{h}^{n+1}:= \sigma_{2}^{n+1} + \xi_{2}^{n+1},$$
$$\delta^{n+1}=p(t^{n+1})-\widetilde{p}^{n+1}
+\widetilde{p}^{n+1}-p_{h}^{n+1}:=\sigma_{\delta}^{n+1}+\xi_{\delta}^{n+1}.$$
Here we assume $\xi_{1}^{0}=0$ and $\xi_{2}^{0}=0$.
From approximation properties, we have 
$$\| \sigma_{1}^{n+1} \|_{0} + \| \sigma_{2}^{n+1}\|_{0}\le C(\epsilon+h^{2}+\frac{\epsilon}{h}), \| \sigma_{1}^{n+1} \|_{1} + \| \sigma_{2}^{n+1}\|_{1}\le C(\sqrt{\epsilon}+h+\sqrt{\frac{\epsilon}{h}}). $$
In the following, we estimate the $\xi_{1}^{n+1}$ and $\xi_{2}^{n+1}$. 
Note that $(\widetilde{\underline{\textbf{u}}}^{n+1},\widetilde{\phi}^{n+1}) \in \textbf{U}_{h} \times Q_{fh}$  fulfills the subsequent equations: for all $\underline{\textbf{v}}_{h}=(\textbf{v}_{h},\psi_{h})\in \textbf{U}_{h},q\in Q_{fh},$
\begin{align}
(\frac{\widetilde{\textbf{u}}^{n+1}-\widetilde{\textbf{u}}^{n}}{\Delta t},\textbf{v}_{h})
&+ 2\nu(\mathbb{D}(\widetilde{\textbf{u}}^{n+1}),\mathbb{D}(\textbf{v}_{h}))_{\Omega_{f}}
+ b(\textbf{v}_{h},\widetilde{p}^{n+1})     
+ \int_{\Gamma} \frac{\nu \alpha \sqrt{d}}{\sqrt{trace(\prod)}} (\widetilde{\textbf{u}}^{n+1} \cdot \tau_{i})(\textbf{v}_{h} \cdot \tau_{i})      \nonumber     \\ 
&=(w_{u}^{n+1},\textbf{v}_{h})
+(\textbf{f}_{f}^{n+1},\textbf{v}_{h})_{\Omega_{f}}
-g \int_{\Gamma} \widetilde{\phi}^{n+1}\textbf{v}_{h} \cdot \textbf{n}_{f},\label{PE1}     
\end{align}
\begin{align}
(\frac{\widetilde{\phi}^{n+1} - \widetilde{\phi}^{n}}{\Delta t},\psi_{h}) 
+ \frac{1}{S_{0}}(\mathbb{K}^{\epsilon}(\textbf{x})\nabla \widetilde{\phi}^{n+1},\nabla 
\psi_{h})_{\Omega_{p}}=&
(w_{\phi}^{n+1},\psi_{h})
+ \frac{1}{S_{0}}(f_{p}^{n+1},\psi_{h})_{\Omega_{p}}  \nonumber \\
&+\frac{1}{S_{0}}\int_{\Gamma} \widetilde{\textbf{u}}^{n+1} \cdot \textbf{n}_{f}\psi_{h}, \label{PE2}
\end{align}
 where 
\begin{align*}
   w_{u}^{n+1} & = \frac{\widetilde{\textbf{u}}^{n+1}-\widetilde{\textbf{u}}^{n}}{\Delta t}-\textbf{u}_{t}(t^{n+1}) \\
   & = [\frac{\widetilde{\textbf{u}}^{n+1}-\widetilde{\textbf{u}}^{n}}{\Delta t}-\frac{\textbf{u}(t^{n+1})-\textbf{u}(t^{n})}{\Delta t}]+[\frac{\textbf{u}(t^{n+1})-\textbf{u}(t^{n})}{\Delta t}-\textbf{u}_{t}(t^{n+1})] \\
   & =w_{u,1}^{n+1} + w_{u,2}^{n+1},
\end{align*}
\begin{align*}
w_{\phi}^{n+1} & = \frac{\widetilde{\phi}^{n+1}-\widetilde{\phi}^{n}}{\Delta t}-\phi_{t}(t^{n+1}) \\
 & = [\frac{\widetilde{\phi}^{n+1}-\widetilde{\phi}^{n}}{\Delta t}-\frac{\phi(t^{n+1})-\phi(t^{n})}{\Delta t}]+[\frac{\phi(t^{n+1})-\phi(t^{n})}{\Delta t}-\phi_{t}(t^{n+1})] \\
 & = w_{\phi,1}^{n+1} + w_{\phi,2}^{n+1}.
\end{align*}
It can be readily confirmed that the following properties of $w_{u,1}^{n+1},w_{u,2}^{n+1},w_{\phi,1}^{n+1}$ and $w_{\phi,2}^{n+1}$.
\begin{align*} 
w_{u,1}^{n+1}=(P_{h}-I)\frac{\textbf{u}(t^{n+1})-\textbf{u}(t^{n})}{\Delta t} = \frac{1}{\Delta t} \int_{t^{n}}^{t^{n+1}}(P_{h}-I)\textbf{u}_{t}(t) dt ,
\end{align*}
so 
\begin{align}\label{wu1}  
\|w_{u,1}^{n+1}\|_{0}^{2}&= \frac{1}{\Delta t^{2}} \int_{\Omega}(\int_{t^{n}}^{t^{n+1}} (P_{h}-I)\textbf{u}_{t}(t) dt)^{2} d\textbf{x} \notag \\
&\le \frac{1}{\Delta t^{2}}\int_{\Omega} (\int_{t^{n}}^{t_{n+1}}  1^{2}  dt)(\int_{t^{n}}^{t_{n+1}} ((P_{h}-I)\textbf{u}_{t}(t))^{2} dt) d\textbf{x}  \notag \\
&\le \frac{1}{\Delta t} \int_{t^{n}}^{t^{n+1}} \|(P_{h}-I)\textbf{u}_{t}\|_{0}^{2} dt.
\end{align}
In a similar way,
$$ w_{u,2}^{n+1}=\frac{\textbf{u}(t^{n+1})-\textbf{u}(t^{n})}{\Delta t}-\textbf{u}_{t}(t^{n+1})=\frac{1}{\Delta t} \int_{t^{n}}^{t^{n+1}} (t^{n+1}-t)\textbf{u}_{tt}(t) dt ,$$
which means
\begin{align}\label{wu2}  
\| w_{u,2}^{n+1}\|_{0}^{2}=\frac{1}{\Delta t^{2}} \int_{\Omega}(\int_{t^{n}}^{t^{n+1}} (t^{n+1}-t)\textbf{u}_{tt}(t)dt)^{2} d\textbf{x} 
\le  \Delta t \int_{t^{n}}^{t^{n+1}} \|\textbf{u}_{tt}\|_{0}^{2} dt.
\end{align}
Similarly, we will obtain the following results
\begin{align}
 \|w_{\phi,1}^{n+1}\|_{0}^{2}
&\le  \frac{1}{\Delta t} \int_{t^{n}}^{t^{n+1}} \|(P_{h}-I)\phi_{t}\|_{0}^{2} dt,\label{wphi1}  \\
 \| w_{\phi,2}^{n+1}\|_{0}^{2}
&\le  \Delta t \int_{t^{n}}^{t^{n+1}} \|\phi_{tt}\|_{0}^{2} dt.\label{wphi2}
\end{align}
Since $P_{h}$ is a bounded operator, we have 
\begin{align}\label{Pb1}
\|\widetilde{\textbf{u}}(t^{n+1})-\widetilde{\textbf{u}}(t^{n}) \|_{1}^{2}
&\le C\| \textbf{u}(t^{n+1})-\textbf{u}(t^{n}) \|_{1}^{2}=C\int_{\Omega} (\int_{t^{n}}^{t^{n+1}} \nabla \textbf{u}_{t}(t)dt)^{2} d\textbf{x}  \notag\\
&\le C\int_{\Omega} (\int_{t^{n}}^{t^{n+1}} \vert\nabla \textbf{u}_{t}(t)\vert^{2} dt\int_{t^{n}}^{t^{n+1}} 1dt) d\textbf{x} 
=C\Delta t \int_{t^{n}}^{t^{n+1}} \| \textbf{u}_{t}\|_{1}^{2} dt, 
\end{align}
\begin{align}\label{Pb2}     
\| \widetilde{\phi}(t^{n+1})-\widetilde{\phi}(t^{n})\|_{1}^{2}
\le C\Delta t \int_{t^{n}}^{t^{n+1}} \| \phi_{t}\|_{1}^{2} dt.
\end{align}
Now we provide the following theorem to demonstrate the error estimate of Algorithm \ref{alg:1}.
\begin{theorem}[Error]
 Under the assumption of Theorem \ref{Stability} for the time step size, the following error estimate holds:
\begin{align} \label{Error}   
\| \textbf{u}(t^{N})-\textbf{u}_{h}^{N} \|_{0}^{2} 
+\| \phi(t^{N})-\phi_{h}^{N} \|_{0}^{2} \le C(T)(\Delta t^{2} +\epsilon^{2} + h^{4} + (\frac{\epsilon}{h})^{2}). 
\end{align}
\end{theorem}
\begin{proof}
Subtracting (\ref{Sa1})-(\ref{Da}) from (\ref{PE1})-(\ref{PE2}), we have 
\begin{align}\label{E1}
(\frac{\xi_{1}^{n+1}-\xi_{1}^{n}}{\Delta t},\textbf{v}_{h})
+ 2\nu(\mathbb{D}(\xi_{1}^{n+1}),\mathbb{D}(\textbf{v}_{h}))_{\Omega_{f}}
&+b(\textbf{v}_{h},\xi_{\delta}^{n+1})     
+ \int_{\Gamma} \frac{\nu \alpha \sqrt{d}}{\sqrt{trace(\prod)}} (\xi_{1}^{n+1} \cdot \tau_{i})(\textbf{v}_{h} \cdot \tau_{i})     \notag \\
&=(w_{u}^{n+1},\textbf{v}_{h})
-g \int_{\Gamma} (\widetilde{\phi}^{n+1}-\phi_{h}^{n})\textbf{v}_{h} \cdot \textbf{n}_{f} , 
\end{align}
\begin{align}\label{E2}
(\frac{\xi_{2}^{n+1}-\xi_{2}^{n}}{\Delta t},\psi_{h})
+ \frac{1}{S_{0}}(\mathbb{K}^{\epsilon}(\textbf{x})\nabla \xi_{2}^{n+1},\nabla 
\psi_{h})_{\Omega_{p}}
= (w_{\phi}^{n+1},\psi_{h})
+\frac{1}{S_{0}}\int_{\Gamma} (\widetilde{\textbf{u}}^{n+1}
 -\textbf{u}_{h}^{n})\cdot \textbf{n}_{f}\psi_{h}.
\end{align}
Taking $\textbf{v}_{h}=2\Delta t \xi_{1}^{n+1}$ in (\ref{E1}) and using the divergence free property, as well as Korn's inequality, we obtain 
\begin{align}\label{E3}
&\|\xi_{1}^{n+1}\|_{0}^{2} 
-\|\xi_{1}^{n}\|_{0}^{2}
 +\|\xi_{1}^{n+1}-\xi_{1}^{n}\|_{0}^{2} 
 +4C_{1}\nu \Delta t  \|\xi_{1}^{n+1} \|_{1}^{2}
 +\int_{\Gamma} \frac{\nu \alpha \sqrt{d}}{\sqrt{trace(\prod)}} (\xi_{1}^{n+1} \cdot \tau_{i})(\xi_{1}^{n+1} \cdot \tau_{i})   \nonumber \\
 & \le 2\Delta t (w_{u}^{n+1},\xi_{1}^{n+1})
 -2g\Delta t \int_{\Gamma} (\widetilde{\phi}^{n+1}-\widetilde{\phi}^{n}+\xi_{2}^{n})\xi_{1}^{n+1} \cdot \textbf{n}_{f} .
\end{align}
Then taking $\psi_{h}=2\Delta t \xi_{2}^{n+1}$ and using the uniform boundedness of $\mathbb{K}^{\epsilon}(\textbf{x}),$ we obtain
\begin{align}\label{E4}
&\|\xi_{2}^{n+1}\|_{0}^{2} 
-\|\xi_{2}^{n}\|_{0}^{2}
 +\|\xi_{2}^{n+1}-\xi_{2}^{n}\|_{0}^{2} 
 +\frac{2\lambda_{min}}{S_{0}}\Delta t\|\xi_{2}^{n+1}\|_{1}^{2}, \nonumber  \\
 &\le 2\Delta t (w_{\phi}^{n+1},\xi_{2}^{n+1})
 + \frac{2\Delta t}{S_{0}}
  \int_{\Gamma} (\widetilde{\textbf{u}}^{n+1}-\widetilde{\textbf{u}}^{n}
 +\xi_{1}^{n})\cdot \textbf{n}_{f}\xi_{2}^{n+1}.
\end{align}
Combining (\ref{E3}) and (\ref{E4}), we have 
\begin{align}\label{E5}
\|\xi_{1}^{n+1}\|_{0}^{2} 
-\|\xi_{1}^{n}\|_{0}^{2}
&+\|\xi_{1}^{n+1}-\xi_{1}^{n}\|_{0}^{2} 
+\|\xi_{2}^{n+1}\|_{0}^{2} 
-\|\xi_{2}^{n}\|_{0}^{2}
+\|\xi_{2}^{n+1}-\xi_{2}^{n}\|_{0}^{2}\notag \\
&+4C_{1}\nu \Delta t  \|\xi_{1}^{n+1}\|_{1}^{2}
+\frac{2\lambda_{min}}{S_{0}}\Delta t \| \xi_{2}^{n+1}\|_{1}^{2} \notag\\
\le &2\Delta t (w_{u}^{n+1},\xi_{1}^{n+1})
+2\Delta t (w_{\phi}^{n+1},\xi_{2}^{n+1})  \nonumber\\
&-2g\Delta t \int_{\Gamma} (\widetilde{\phi}^{n+1}-\widetilde{\phi}^{n})\xi_{1}^{n+1}\cdot \textbf{n}_{f} 
+\frac{2\Delta t}{S_{0}}\int_{\Gamma} (\widetilde{\textbf{u}}^{n+1}-\widetilde{\textbf{u}}^{n}
)\cdot \textbf{n}_{f}\xi_{2}^{n+1} \nonumber\\
&-2g\Delta t \int_{\Gamma} \xi_{2}^{n} \xi_{1}^{n+1}\cdot \textbf{n}_{f}
+\frac{2\Delta t}{S_{0}}\int_{\Gamma} \xi_{1}^{n}\cdot \textbf{n}_{f}\xi_{2}^{n+1}. 
\end{align}
By applying $\text {Hölder}$, $\text {Poincaré}$, Young inequalities and (\ref{wu1})-(\ref{wphi2}), we can simplify the first two terms on the right-hand side of (\ref{E5}). This simplification results in 
\begin{align}\label{E6}
2\Delta t &(w_{u}^{n+1},\xi_{1}^{n+1})
+2\Delta t (w_{\phi}^{n+1},\xi_{2}^{n+1}) \nonumber\\
\le & \frac{11C_{1}\nu}{4}\Delta t \| \xi_{1}^{n+1}\|_{1}^{2}  
+\frac{7\lambda_{min}}{6S_{0}}\Delta t \| \xi_{2}^{n+1}\|_{1}^{2}
+ \frac{4C_{p}^{2}}{11C_{1}\nu}\Delta t \|w_{u}^{n+1}\|_{0}^{2}
+ \frac{6S_{0}\widetilde{C}_{p}^{2}}{7\lambda_{min}}\Delta t \|w_{\phi}^{n+1}\|_{0}^{2}    \nonumber\\
\le & \frac{11C_{1}\nu}{4}\Delta t \| \xi_{1}^{n+1}\|_{1}^{2}
+\frac{7\lambda_{min}}{6S_{0}}\Delta t \| \xi_{2}^{n+1}\|_{1}^{2}\nonumber\\
&+\frac{4C_{p}^{2}}{11C_{1}\nu}\int_{t^{n}}^{t^{n+1}} \|(P_{h}-I)\textbf{u}_{t}\|_{0}^{2} dt
+\frac{4C_{p}^{2}}{11C_{1}\nu} \Delta t^{2} \int_{t^{n}}^{t^{n+1}} \|\textbf{u}_{tt}\|_{0}^{2} dt     \nonumber\\
&+\frac{6S_{0}\widetilde{C}_{p}^{2}}{7\lambda_{min}} \int_{t^{n}}^{t^{n+1}} \|(P_{h}-I)\phi_{t}\|_{0}^{2} dt  
+\frac{6S_{0}\widetilde{C}_{p}^{2}}{7\lambda_{min}}   \Delta t^{2} \int_{t^{n}}^{t^{n+1}} \|\phi_{tt}\|_{0}^{2} dt.    
\end{align}
For the second two terms on the right-hand side of (\ref{E5}), by using $\text {Hölder}$ inequality, $\text {Poincaré}$ inequality, Young inequality and (\ref{Pb1})-(\ref{Pb2}), it gives that
\begin{align}\label{E7}
&-2g\Delta t \int_{\Gamma} (\widetilde{\phi}^{n+1}-\widetilde{\phi}^{n})\xi_{1}^{n+1}\cdot \textbf{n}_{f} 
+\frac{2\Delta t}{S_{0}}\int_{\Gamma} (\widetilde{\textbf{u}}^{n+1}-\widetilde{\textbf{u}}^{n}
)\cdot \textbf{n}_{f}\xi_{2}^{n+1} \notag\\
 &\le \frac{C_{1}\nu}{4}\Delta t \| \xi_{1}^{n+1}\|_{1}^{2}  
+\frac{\lambda_{min}}{12S_{0}}\Delta t \| \xi_{2}^{n+1}\|_{1}^{2}    \notag\\
& \quad+\frac{4g^{2}C_{t}^{2}\widetilde{C}_{t}^{2}C_{p}^{2}\widetilde{C}_{p}^{2}}{C_{1}\nu}\Delta t\| \widetilde{\phi}^{n+1}-\widetilde{\phi}^{n}\|_{1}^{2}
+\frac{12C_{t}^{2}\widetilde{C}_{t}^{2}C_{p}^{2}\widetilde{C}_{p}^{2}}{S_{0}\lambda_{min}} \Delta t \| \widetilde{\textbf{u}}^{n+1}-\widetilde{\textbf{u}}^{n}\|_{1}^{2} \notag\\
&\le \frac{C_{1}\nu}{4}\Delta t \| \xi_{1}^{n+1}\|_{1}^{2}  
+\frac{\lambda_{min}}{12S_{0}}\Delta t \| \xi_{2}^{n+1}\|_{1}^{2} \notag\\
&\quad +\frac{4Cg^{2}C_{t}^{2}\widetilde{C}_{t}^{2}C_{p}^{2}\widetilde{C}_{p}^{2}}{C_{1}\nu}\Delta t^{2} \int_{t^{n}}^{t^{n+1}} \| \phi_{t}\|_{1}^{2} dt
+\frac{12CC_{t}^{2}\widetilde{C}_{t}^{2}C_{p}^{2}\widetilde{C}_{p}^{2}}{S_{0}\lambda_{min}} \Delta t^{2}\int_{t^{n}}^{t^{n+1}} \| \textbf{u}_{t}\|_{1}^{2} dt.
\end{align}
For the last two terms on the right-hand side of (\ref{E5}), using $\text {Hölder}$, trace, and Young inequalities, we have 
\begin{align}\label{E8}
-&2g\Delta t \int_{\Gamma} \xi_{2}^{n} \xi_{1}^{n+1}\cdot \textbf{n}_{f}
+\frac{2\Delta t}{S_{0}}\int_{\Gamma} \xi_{1}^{n}\cdot \textbf{n}_{f}\xi_{2}^{n+1}\nonumber \\
\le& \frac{C_{1}\nu}{2}\Delta t \| \xi_{1}^{n+1}\|_{1}^{2}
+\frac{\lambda_{min}}{2S_{0}}\Delta t \| \xi_{2}^{n+1}\|_{1}^{2} 
+\frac{\lambda_{min}}{4S_{0}}\Delta t \| \xi_{2}^{n}\|_{1}^{2} 
+\frac{C_{1}\nu}{2}\Delta t \| \xi_{1}^{n}\|_{1}^{2}  \nonumber\\
&+\frac{g^{2}C_{t}^{2}\widetilde{C}_{t}^{2}\sqrt{S_{0}}}{\sqrt{2C_{1}\nu \lambda_{min}}}\Delta t(\|\xi_{1}^{n+1}\|_{0}^{2}+\|\xi_{2}^{n}\|_{0}^{2})
+\frac{C_{t}^{2}\widetilde{C}_{t}^{2}}{2\sqrt{S_{0}^{3}C_{1}\nu \lambda_{min}}}\Delta t (\|\xi_{1}^{n}\|_{0}^{2}+\|\xi_{2}^{n+1}\|_{0}^{2}).
\end{align}
Considering (\ref{E6})-(\ref{E8}) and (\ref{E5}), we can get 
\begin{align} \label{E9}
\|\xi_{1}^{n+1}\|_{0}^{2} 
-&\|\xi_{1}^{n}\|_{0}^{2}
+\|\xi_{1}^{n+1}-\xi_{1}^{n}\|_{0}^{2} 
+\|\xi_{2}^{n+1}\|_{0}^{2} 
-\|\xi_{2}^{n}\|_{0}^{2}
+\|\xi_{2}^{n+1}-\xi_{2}^{n}\|_{0}^{2}  \notag \\
&+\frac{C_{1}\nu}{2}\Delta t ( \|\xi_{1}^{n+1}\|_{1}^{2} 
  -\|\xi_{1}^{n}\|_{1}^{2} )
+\frac{\lambda_{min}}{4S_{0}}\Delta t 
(\| \xi_{2}^{n+1}\|_{1}^{2}
  -\| \xi_{2}^{n}\|_{1}^{2})   \nonumber\\
\le & \frac{g^{2}C_{t}^{2}\widetilde{C}_{t}^{2}\sqrt{S_{0}}}{\sqrt{2C_{1}\nu \lambda_{min}}}\Delta t(\|\xi_{1}^{n+1}\|_{0}^{2}+\|\xi_{2}^{n}\|_{0}^{2})
+\frac{C_{t}^{2}\widetilde{C}_{t}^{2}}{2\sqrt{S_{0}^{3}C_{1}\nu \lambda_{min}}}\Delta t (\|\xi_{1}^{n}\|_{0}^{2}+\|\xi_{2}^{n+1}\|_{0}^{2})  \nonumber\\
&+\frac{4C_{p}^{2}}{11C_{1}\nu}\int_{t^{n}}^{t^{n+1}} \|(P_{h}-I)\textbf{u}_{t}\|_{0}^{2} dt
+\frac{4C_{p}^{2}}{11C_{1}\nu} \Delta t^{2} \int_{t^{n}}^{t^{n+1}} \|\textbf{u}_{tt}\|_{0}^{2} dt     \nonumber\\
&+\frac{6S_{0}\widetilde{C}_{p}^{2}}{7\lambda_{min}} \int_{t^{n}}^{t^{n+1}} \|(P_{h}-I)\phi_{t}\|_{0}^{2} dt  
+\frac{6S_{0}\widetilde{C}_{p}^{2}}{7\lambda_{min}}   \Delta t^{2} \int_{t^{n}}^{t^{n+1}} \|\phi_{tt}\|_{0}^{2} dt. \notag \\
&+\frac{4Cg^{2}C_{t}^{2}\widetilde{C}_{t}^{2}C_{p}^{2}\widetilde{C}_{p}^{2}}{C_{1}\nu}\Delta t^{2} \int_{t^{n}}^{t^{n+1}} \| \phi_{t}\|_{1}^{2} dt
+\frac{12CC_{t}^{2}\widetilde{C}_{t}^{2}C_{p}^{2}\widetilde{C}_{p}^{2}}{S_{0}\lambda_{min}} \Delta t^{2}\int_{t^{n}}^{t^{n+1}} \| \textbf{u}_{t}\|_{1}^{2} dt.
\end{align}
Summing (\ref{E9}) from $n=0$ to $N-1$, we have 
\begin{align} \label{E10}
\|\xi_{1}^{N}\|_{0}^{2}
+&\|\xi_{2}^{N}\|_{0}^{2}
+\sum\limits_{n=0}^{N-1}(\|\xi_{1}^{n+1}-\xi_{1}^{n}\|_{0}^{2} 
+\|\xi_{2}^{n+1}-\xi_{2}^{n}\|_{0}^{2})
+\frac{C_{1}\nu}{2}\Delta t \|\xi_{1}^{N} \|_{1}^{2} 
+\frac{\lambda_{min}}{4S_{0}}\Delta t \| \xi_{2}^{N}\|_{1}^{2}\notag \\
\le \widetilde{C}\Delta t & \sum\limits_{n=0}^{N-1} (\|\xi_{1}^{n+1}\|_{0}^{2}+\|\xi_{2}^{n+1}\|_{0}^{2}) 
+\|\xi_{1}^{0}\|_{0}^{2}
+\|\xi_{2}^{0}\|_{0}^{2}
+\frac{C_{1}\nu}{2}\Delta t  \|\xi_{1}^{0} \|_{1}^{2} 
+\frac{\lambda_{min}}{4S_{0}}\Delta t \| \xi_{2}^{0}\|_{1}^{2}   \notag\\
&+\frac{4C_{p}^{2}}{11C_{1}\nu}  \|(P_{h}-I)\textbf{u}_{t}\|_{L^{2}(0,T;L^{2})}^{2}
 +\frac{4C_{p}^{2}}{11C_{1}\nu}  \Delta t^{2}  \|\textbf{u}_{tt}\|_{L^{2}(0,T;L^{2})}^{2} \notag     \\
 &+\frac{6S_{0}\widetilde{C}_{p}^{2}}{7\lambda_{min}}  \|(P_{h}-I)\phi_{t}\|_{L^{2}(0,T;L^{2})}^{2}
+\frac{6S_{0}\widetilde{C}_{p}^{2}}{7\lambda_{min}} \Delta t^{2}\|\phi_{tt}\|_{L^{2}(0,T;L^{2})}^{2}\notag \\
&+\frac{4Cg^{2}C_{t}^{2}\widetilde{C}_{t}^{2}C_{p}^{2}\widetilde{C}_{p}^{2}}{C_{1}\nu}\Delta t^{2}\|\phi_{t}\|_{L^{2}(0,T;H^{1})}^{2}
+\frac{12CC_{t}^{2}\widetilde{C}_{t}^{2}C_{p}^{2}\widetilde{C}_{p}^{2}}{S_{0}\lambda_{min}} \Delta t^{2}\|\textbf{u}_{t}\|_{L^{2}(0,T;H^{1})}^{2}.
\end{align}
When $\widetilde{C}\Delta t  \textless 1$, using Gronwall inequality we can get the following inequality,
\begin{align} \label{E11}
\|\xi_{1}^{N}\|_{0}^{2}
+&\|\xi_{2}^{N}\|_{0}^{2}
+\sum\limits_{n=0}^{N-1}(\|\xi_{1}^{n+1}-\xi_{1}^{n}\|_{0}^{2} 
+\|\xi_{2}^{n+1}-\xi_{2}^{n}\|_{0}^{2})
+\frac{C_{1}\nu}{2}\Delta t \|\xi_{1}^{N} \|_{1}^{2}
+\frac{\lambda_{min}}{4S_{0}}\Delta t \| \xi_{2}^{N}\|_{1}^{2} \notag \\
\le& C(T)(\|\xi_{1}^{0}\|_{0}^{2}
+\|\xi_{2}^{0}\|_{0}^{2}
+\frac{C_{1}\nu}{2}\Delta t  \|\xi_{1}^{0} \|_{1}^{2} 
+\frac{\lambda_{min}}{4S_{0}}\Delta t \| \xi_{2}^{0}\|_{1}^{2}   \nonumber \\
&+\frac{4C_{p}^{2}}{11C_{1}\nu}  \|(P_{h}-I)\textbf{u}_{t}\|_{L^{2}(0,T;L^{2})}^{2}
 +\frac{4C_{p}^{2}}{11C_{1}\nu}  \Delta t^{2}  \|\textbf{u}_{tt}\|_{L^{2}(0,T;L^{2})}^{2} \notag     \\
 &+\frac{6S_{0}\widetilde{C}_{p}^{2}}{7\lambda_{min}}  \|(P_{h}-I)\phi_{t}\|_{L^{2}(0,T;L^{2})}^{2}
+\frac{6S_{0}\widetilde{C}_{p}^{2}}{7\lambda_{min}}\Delta t^{2} \|\phi_{tt}\|_{L^{2}(0,T;L^{2})}^{2}\notag \\
&+\frac{4Cg^{2}C_{t}^{2}\widetilde{C}_{t}^{2}C_{p}^{2}\widetilde{C}_{p}^{2}}{C_{1}\nu}\Delta t^{2}\|\phi_{t}\|_{L^{2}(0,T;H^{1})}^{2}
+\frac{12CC_{t}^{2}\widetilde{C}_{t}^{2}C_{p}^{2}\widetilde{C}_{p}^{2}}{S_{0}\lambda_{min}} \Delta t^{2}\|\textbf{u}_{t}\|_{L^{2}(0,T;H^{1})}^{2}).
\end{align}
Therefore, 
\begin{align} \label{E12}
&\|\xi_{1}^{N}\|_{0}^{2}
+\|\xi_{2}^{N}\|_{0}^{2} \notag \\
\le& C(T)(\|\xi_{1}^{0}\|_{0}^{2}
+\|\xi_{2}^{0}\|_{0}^{2}
+\frac{C_{1}\nu}{2}\Delta t  \|\xi_{1}^{0} \|_{1}^{2} 
+\frac{\lambda_{min}}{4S_{0}}\Delta t \| \xi_{2}^{0}\|_{1}^{2}   \nonumber \\
&+\frac{4C_{p}^{2}}{11C_{1}\nu}  \|(P_{h}-I)\textbf{u}_{t}\|_{L^{2}(0,T;L^{2})}^{2}
+\frac{6S_{0}\widetilde{C}_{p}^{2}}{7\lambda_{min}}  \|(P_{h}-I)\phi_{t}\|_{L^{2}(0,T;L^{2})}^{2} \notag     \\
 & +\frac{4C_{p}^{2}}{11C_{1}\nu}  \Delta t^{2}  \|\textbf{u}_{tt}\|_{L^{2}(0,T;L^{2})}^{2}
+\frac{6S_{0}\widetilde{C}_{p}^{2}}{7\lambda_{min}}\Delta t^{2} \|\phi_{tt}\|_{L^{2}(0,T;L^{2})}^{2}\notag \\
&+\frac{4Cg^{2}C_{t}^{2}\widetilde{C}_{t}^{2}C_{p}^{2}\widetilde{C}_{p}^{2}}{C_{1}\nu}\Delta t^{2}\|\phi_{t}\|_{L^{2}(0,T;H^{1})}^{2}
+\frac{12CC_{t}^{2}\widetilde{C}_{t}^{2}C_{p}^{2}\widetilde{C}_{p}^{2}}{S_{0}\lambda_{min}} \Delta t^{2}\|\textbf{u}_{t}\|_{L^{2}(0,T;H^{1})}^{2}).
\end{align}
Finally, using (\ref{P1})-(\ref{P2}), (\ref{E12}), the triangle inequality and the error assumpitions for initial data, we can easily obtain the final result (\ref{Error}).
\end{proof}

\section{Numerical experiments}
In this part, we substantiate our analysis through a series of numerical experiments.
For the problems discussed in this paper, constructing an exact solution is challenging. Consequently, we use numerical solution on a fine grid as a substitute for the exact solution. Moreover, all the code presented in this paper is developed using the FreeFEM++ software package \cite{H}.

\subsection{Implementation}
In this context, we elaborate on the specific implementation process of the experiment and define some commonly used symbols.
For ease of comparing different approaches, we introduce the following symbols: FEM represents the standard finite element method, while MsFEM represents the multiscale finite element method. 

To calculate errors, we first obtain a reference solution using the FEM. Specifically, the Stokes equation utilizes Taylor-Hood elements, and the Darcy equation employs P2 elements. Moreover, when implementing code with FreeFEM++, by loading libraries for parallel computing, such as PETSc or MUMPS, we can efficiently obtain the reference solution.
 When calculating the reference solution, we use $h_{D}$ to represent the mesh size in the porous media region, while $h_{S}$ represents the mesh size in the free flow region.
Due to the multiscale phenomenon in the Darcy region, we choose $h_{D}$ = 1/2048, while in the interior of the Stokes region, there is no multiscale phenomenon, so we choose $h_{S}$ = 1/512. Simultaneously, we fix the time step as $\Delta t=0.01$.


 $\bullet$ MsFEM-ImEx: For Algorithm \ref{alg:1}(MsFEM-ImEx), we employ MINI elements for the finite element space in the Stokes problem. The multiscale finite element space in the Darcy problem is spanned by the multiscale basis functions obtained through offline process(Pre-work). In the following, $h$ represents the mesh size (coarse mesh) when solving the Stokes-Darcy model, and $h_{fine}$ represents the mesh size (fine mesh) when solving the multiscale basis functions. Additionally, $nsplit=h/h_{fine}$.

$\bullet$ FEM-ImEx: To compare with the FEM, we use the finite element method to compute both the Stokes and Darcy regions. In particular, we choose MINI elements for the Stokes part and P1 elements for the Darcy part as the finite element spaces. 
Similarly, we use FEM-ImEx to represent the case where both regions are solved using the FEM.


For the sake of symbol simplification, we represent the solutions of the MsFEM-ImEx scheme as $\left({\textbf{u}}^{h,n}_{1},p_{1}^{h,n},\phi^{h,n}_{1}\right)$. Subsequently, we use the following symbols to represent the error between the numerical solution and the reference solution. 
$$
e_{1,\textbf{u}}^{h,n}={\textbf u}^{h,n}_{1}-{\textbf u}(t^{n}), \quad e_{1,\phi}^{h,n}=\phi_{1}^{h,n}-\phi(t^{n}).
$$
Here $u(t^n)$ and $\phi(t^n)$ represent the reference solution.

In the FEM-ImEx scheme, we denote its solution as $\left(\textbf{u}^{h,n}_{2},p_{2}^{h,n},\phi_{2}^{h,n}\right)$, and then utilize the following symbols:
$$
e_{2,\textbf{u}}^{h,n}={\textbf u}^{h,n}_{2}-{\textbf u}(t^{n}), \quad e_{2,\phi}^{h,n}=\phi_{2}^{h,n}-\phi(t^{n}).
$$

Referring to \cite{MuZ}, we propose the method in this paper for calculating the convergence order with respect to the time step $\Delta t$. So we define:
$$
\rho_{v,h,i}=\frac{\left\|v^{h,n}_{1,\Delta t}-v^{h,n}_{1,\frac{\Delta t}{2}}\right\|_{i}}{\left\|v^{h,n}_{1,\frac{\Delta t}{2}}-v^{h,n}_{1,\frac{\Delta t}{4}}\right\|_{i}} 
\approx \frac{4^\gamma-2^\gamma}{2^\gamma-1}=2^\gamma,
$$
where $v=\textbf{u}$ or $\phi$, and $i=0$ or 1. Particularly, $\rho\approx 2$ for $\gamma=1$, when the corresponding order of convergence in time is $O(t)$.

\subsection{Numerical results}
\begin{example}\label{example1}
Let $\Omega_{p}=[0,1]\times[0,1]$, $\Omega_{f}=[0,1]\times[1,2]$ and $\Gamma=(0,1)\times \{1\}$.
In the Darcy region, we set the right hand side $f_{p}=t$, the initial value $\phi(0)=0$ and $\phi=0$ on $\Gamma_{p}$. In addition, we assume $\mathbb{K}^{\epsilon}(\textbf{x})=K^{\epsilon}(\textbf{x})I$, where 
$$ K^{\epsilon}(\textbf{x})=\frac{1}{(2+P\sin(\frac{2\pi x}{\epsilon}))(2+P\cos(\frac{2\pi y}{\epsilon}))},\ (P=1.5).$$
Here P is a parameter that controls the amplitude, and $K^{\epsilon}(\textbf{x})$ is
separable in space.  
Next, we consider the Stokes region. We set the right-hand side $\textbf{f}_{f}=[\textbf{f}_{f1},\textbf{f}_{f2}]$, the initial value, and boundary conditions on $\Gamma_{f}$ as follows:
\begin{align*}
\textbf{f}_{f1} &= (1-2x)(y-1)(\cos(t)-\sin(t)), \\
\textbf{f}_{f2} &= -(x(x-1)+(y-1)^2)\sin(t)+(x(1-x)+(y+1)(y-3))\cos(t), \\
\textbf{u}(0) &= [(1-2x)(y-1),x(x-1)+(y-1)^2], \\
\textbf{u} &= [(1-2x)(y-1)\cos(t),x(x-1)+(y-1)^2\cos(t)] \ \text{on} \ \Gamma_{f}.
\end{align*}
We set T = 1, and all physical parameters $S_{0},\nu ,\alpha, g$ equal to 1.
\end{example}
\begin{table}[!h]
\caption{Convergence performance of the MsFEM-ImEx scheme at $t^n = 1.0$, with varying h but fixed $\Delta t=0.01$, $h_{fine}=1/2048$ and $\epsilon=0.0085$.}
\label{table1}
\begin{tabular}{c|c|c|c|c|c|c|c|c}	
\hline\hline$h$ & $\left\|e_{1,\textbf{u}}^{h,n}\right\|_0$ & Rate & $\left\|e_{1,\textbf{u}}^{h,n}\right\|_1$ & Rate & $\left\|e_{1,\phi}^{h,n}\right\|_0$ & Rate & $\left\|e_{1,\phi}^{h,n}\right\|_1$ & Rate \\
\hline $1 / 2$ & $3.51 \mathrm{E}-02$ & & $4.40 \mathrm{E}-01$ & & $6.26 \mathrm{E}-02$ & & $4.36 \mathrm{E}-01$ & \\
		$1 / 4$ & $8.92 \mathrm{E}-03$ & 2.0 & $1.99 \mathrm{E}-01$ & 1.1 & $1.82 \mathrm{E}-02$ & 1.8 & $2.39 \mathrm{E}-01$ & 0.87 \\
		$1 / 8$ & $2.22 \mathrm{E}-03$ & 2.0 & $9.54 \mathrm{E}-02$ & 1.1 & $6.77 \mathrm{E}-03$ & 1.4 & $1.41 \mathrm{E}-01$ & 0.77 \\
		$1 / 16$ & $5.92 \mathrm{E}-04$ & 1.9 & $4.70 \mathrm{E}-02$ & 1.0 & $6.07 \mathrm{E}-03$ & 0.16 & $1.22 \mathrm{E}-01$ & 0.21 \\
		$1 / 32$ & $2.29 \mathrm{E}-04$ & 1.4 & $2.35 \mathrm{E}-02$ & 1.0 & $1.00 \mathrm{E}-02$ & -0.72 & $1.55 \mathrm{E}-01$ & -0.34 \\
		$1 / 64$ & $2.39 \mathrm{E}-04$ & -0.064 & $1.20 \mathrm{E}-02$ & 1.0 & $1.92 \mathrm{E}-02$ & -0.93 & $2.21 \mathrm{E}-01$ & -0.51 \\
\hline
\end{tabular}
\end{table}
In table \ref{table1}, we show the convergence orders in space for the MsFEM-ImEx scheme with fixed $\epsilon=0.0085, h_{fine}=1/2048,$ time step $\Delta t=0.01$ and varying mesh size $h=1/2, 1/4, 1/8, 1/16,$ $1/32, 1/64$.
When $h\textgreater \epsilon$, it can be observed that the convergence order in space for the $L^2$ norm of the error for $\phi$ decreases from $O(h^2)$ to $O(h^{-1})$. This observation is consistent with the theoretical analysis of $L^2$ error estimation. Meanwhile the $H^1$ norm of the error for $\phi$ exhibits a decrease from $O(h)$ to $O(h^{-1/2})$.
Next, we observe the error in the Stokes region. When the grid size $h$ is relatively large, the convergence order of the $L^{2}$ norm error of $\phi$ in space approaches $O(h^2)$, and concurrently the convergence order of the $L^2$ norm error of $\textbf{u}$ in space remains at $O(h^2)$.
However, as the $L^2$ error order of $\phi$ gradually degrades to $O(h^{-1})$, influenced by the Darcy region, the $L^2$ convergence order of $\textbf{u}$ also shows a decline.
Additionally, the error in $H^{1}$ norm of $\textbf{u}$ is essentially not influenced, and the convergence order in spatial terms remains at $O(h)$.
\begin{table}[!h]
\caption{Error of the FEM-ImEx scheme at $t^n=1.0$, with varying $h$ but fixed $\Delta t=0.01$ and $\epsilon=0.0085$.}
\label{table2}
	\begin{tabular}{c|c|c|c|c|c|c|c|c}
\hline\hline$h$ & $\left\|e_{2,\textbf{u}}^{h,n}\right\|_0$ & Rate & $\left\|e_{2,\textbf{u}}^{h,n}\right\|_1$ & Rate & $\left\|e_{2,\phi}^{h,n}\right\|_0$ & Rate & $\left\|e_{2,\phi}^{h,n}\right\|_1$ & Rate \\
		\hline $1 / 2$ & $3.51 \mathrm{E}-02$ & & $4.40 \mathrm{E}-01$ & & $1.21 \mathrm{E}-01$ & & $4.56 \mathrm{E}-01$ & \\
		$1 / 4$ & $8.96 \mathrm{E}-03$ & 2.0 & $1.99 \mathrm{E}-01$ & 1.1 & $8.33 \mathrm{E}-02$ & 0.53 & $3.44 \mathrm{E} -01$ & 0.41 \\
		$1 / 8$ & $2.30 \mathrm{E}-03$ & 2.0 & $9.55 \mathrm{E}-02$ & 1.1 & $7.06 \mathrm{E}-02$ & 0.24 & $3.54 \mathrm{E}-01$ & -0.040 \\
		$1 / 16$ & $8.48 \mathrm{E}-04$ & 1.4 & $4.73 \mathrm{E}-02$ & 1.0 & $5.49 \mathrm{E}-02$ & 0.36 & $3.81 \mathrm{E}-01$ & -0.11 \\
		$1 / 32$ & $6.02 \mathrm{E}-04$ & 0.49 & $2.41 \mathrm{E}-02$ & 0.97 & $4.86 \mathrm{E}-02$ & 0.18 & $3.83 \mathrm{E}-01$ & -0.010 \\
		$1 / 64$ & $5.66 \mathrm{E}-04$ & 0.090 & $1.31 \mathrm{E}-02$ & 0.88 & $5.18 \mathrm{E}-02$ & -0.092 & $3.84 \mathrm{E}-01$ & -0.0019 \\
		\hline
	\end{tabular}
\end{table}

To compare with the FEM, we present the spatial errors of the FEM-ImEx scheme in Table \ref{table2}. 
Primarily, it is evident that the $L^2$ and $H^1$ errors of $\phi$ in the FEM-ImEx scheme are significantly higher than those of the MsFEM-ImEx scheme in table \ref{table1}. Additionally, in these two schemes, the $L^2$ error of u is smaller in the MsFEM-ImEx scheme, and $H^1$ error is essentially the same.
Therefore, under the same coarse grid size $h$ and time step $\Delta t$, the accuracy of MsFEM-ImEx is higher. 
 It is important to note that we provide the errors of the FEM-ImEx scheme for reference only, and we do not focus on their variations or error orders.
\begin{table}[!h]
\caption{Error and convergence order of the MsFEM-ImEx scheme at $t^n=1.0$, with fixed $\Delta t=0.01$, $\frac{\epsilon}{h}=0.32$ and $nsplit=32$.}
\label{table3}
\resizebox{\linewidth}{!}{ 
\begin{tabular}{c|c|c|c|c|c|c|c|c|c}
\hline\hline$h$ & $\epsilon$ & $\left\|e_{1,\textbf{u}}^{h,n}\right\|_0$ & Rate & $\left\|e_{1,\textbf{u}}^{h,n}\right\|_1$ & Rate & $\left\|e_{1,\phi}^{h,n}\right\|_0$ & Rate & $\left\|e_{1,\phi}^{h,n}\right\|_1$ & Rate \\
			\hline $1 / 8$ & 0.04 & $2.25 \mathrm{E}-03$ & & 0.055436 & & 0.0184429 & & 0.213888 & \\
			$1 / 16$ & 0.02 & $6.21 \mathrm{E}-04$ & 1.9 & 0.0470126 & 1.0 & 0.0144562 & 0.35 & 0.180426 & 0.25 \\
			$1 / 32$ & 0.01 & $2.57 \mathrm{E}-04$ & 1.3 & 0.0234862 & 1.0 & 0.0135232 & 0.096 & 0.171474 & 0.073 \\
			$1 / 64$ & 0.005 & $1.77 \mathrm{E}-04$ & 0.54 & 0.0118952 & 0.98 & 0.0132466 & 0.03 & 0.169347 & 0.018 \\
\hline
\end{tabular}
}
\end{table}

Analyzing the data in Table \ref{table3}, where we fix $\Delta t=0.01$, $\frac{\epsilon}{h}=0.32$ and $nsplit=32$, we can observe that in space, the $L^2$ and $H^1$ errors of $\phi$ essentially remain unchanged. This indicates that the term $O(\frac{\epsilon}{h})$ is included in our error estimates, consistent with the theoretical result (\ref{Error}). At the same time, $\textbf{u}$ is impacted by the Darcy region, resulting in a reduction in the convergence order.
\begin{table}[!h]
\caption{Convergence performance of the MsFEM-ImEx scheme at $t^n=1.0$, with varying $\Delta t$ but fixed $h=1/8$, $nsplit=8$ and $\epsilon=0.0085$.}
\label{table4}
\resizebox{\linewidth}{!}{                
\begin{tabular}{c|c|r|c|r|c|c|c|c}
\hline\hline$\Delta t$ & $\left\|\mathbf{u}_{1,\Delta t}^{h,n}-\mathbf{u}_{1,\frac{\Delta t}{2}}^{h,n}\right\|_{0}$ & $\rho_{\textbf{u},h,0}$ & $\left\|\mathbf{u}_{1,\Delta t}^{h,n}-\mathbf{u}_{1,\frac{\Delta t}{2}}^{h,n}\right\|_{1}$ & $\rho_{\textbf{u},h,1}$ & $\left\|\phi_{1,\Delta t}^{h,n}-\phi_{1,\frac{\Delta t}{2}}^{h,n}\right\|_{0}$ & $\rho_{\phi,h,0}$ & $\left\|\phi_{1,\Delta t}^{h,n}-\phi_{1,\frac{\Delta t}{2}}^{h,n}\right\|_{1}$ & $\rho_{\phi,h,1}$ \\
\hline0.1 & $1.66 \mathrm{E}-05$ & 1.9 & $1.66 \mathrm{E}-04$ & 1.9 & $1.48 \mathrm{E}-03$ & 2.0 & $6.73 \mathrm{E}-03$ & 2.0 \\
				0.05 & $8.64 \mathrm{E}-06$ & 2.0 & $8.61 \mathrm{E}-05$ & 2.0 & $7.43 \mathrm{E}-04$ & 2.0 & $3.42 \mathrm{E}-03$ & 2.0 \\
				0.025 & $4.41 \mathrm{E}-06$ & 2.0 & $4.40 \mathrm{E}-05$ & 2.0 & $3.72 \mathrm{E}-04$ & 2.0 & $1.72 \mathrm{E}-03$ & 2.0 \\
				0.0125 & $2.23 \mathrm{E}-06$ & 2.0 & $2.23 \mathrm{E}-05$ & 2.0 & $1.86 \mathrm{E}-04$ & 2.0 & $8.62 \mathrm{E}-04$ & 2.0 \\
				0.00625 & $1.12 \mathrm{E}-06$ & & $1.12 \mathrm{E}-05$ & & $9.29 \mathrm{E}-05$ & & $4.32 \mathrm{E}-04$ & \\
\hline
\end{tabular}
}
\end{table}

To examine the convergence order with respect to time, we fix $h=1/8$, $nsplit=8$ and $\epsilon=0.0085,$ while varying $\Delta t$ as 0.1, 0.05, 0.025, 0.0125, and 0.0625.
The value of $\rho_{\textbf{u},h,0}$, $\rho_{\textbf{u},h,1}$,  $\rho_{\phi,h,0}$ and 
$\rho_{\phi,h,1}$ are all approximately 2 in table \ref{table4}. It is suggested that the orders of convergence in time for $\textbf{u}$ and $\phi$ should be $O(t)$, implying that the error estimates for $\textbf{u}$ and $\phi$ in $L^2$ norm are optimal. 
\begin{example}
In this example, we choose the permeability coefficient 
$$ K^{\epsilon}(\textbf{x})=\frac{1}{4.0+P(sin(\frac{2\pi x}{\epsilon})+ sin(\frac{2\pi y}{\epsilon}))},\ (P=1.8).$$
Here $K^{\epsilon}(\textbf{x})$ is inseparable in space.
All configurations are consistent with those in Example \ref{example1}, with the exception of $ K^{\epsilon}(\textbf{x})$.
\end{example}
\begin{table}[!h]
\caption{Convergence performance of the MsFEM-ImEx scheme at $t^n = 1.0$, with varying h but fixed $\Delta t=0.01$, $h_{fine}=1/2048$ and $\epsilon=0.0085$.}
\label{table5}
\begin{tabular}{c|c|c|c|c|c|c|c|c}	
\hline\hline$h$ & $\left\|e_{1,\textbf{u}}^{h,n}\right\|_0$ & Rate & $\left\|e_{1,\textbf{u}}^{h,n}\right\|_1$ & Rate & $\left\|e_{1,\phi}^{h,n}\right\|_0$ & Rate & $\left\|e_{1,\phi}^{h,n}\right\|_1$ & Rate \\
\hline $1 / 2$ & $3.51 \mathrm{E}-02$ & & $4.40 \mathrm{E}-01$ & & $7.61 \mathrm{E}-02$ & & $5.24 \mathrm{E}-01$ & \\
			$1 / 4$ & $8.99 \mathrm{E}-03$ & 2.0 & $1.99 \mathrm{E}-01$ & 1.1 & $2.14 \mathrm{E}-02$ & 1.8 & $2.86 \mathrm{E}-01$ & 0.87 \\
			$1 / 8$ & $2.27 \mathrm{E}-03$ & 2.0 & $9.59 \mathrm{E}-02$ & 1.1 & $6.76 \mathrm{E}-03$ & 1.7 & $1.60 \mathrm{E}-01$ & 0.84 \\
			$1 / 16$ & $6.11 \mathrm{E}-04$ & 1.9 & $4.73 \mathrm{E}-02$ & 1.0 & $4.47 \mathrm{E}-03$ & 0.60 & $1.20 \mathrm{E}-01$ & 0.41 \\
			$1 / 32$ & $2.19 \mathrm{E}-04$ & 1.5 & $2.37 \mathrm{E}-02$ & 1.0 & $6.74 \mathrm{E}-03$ & -0.59 & $1.42 \mathrm{E}-01$ & -0.24 \\
			$1 / 64$ & $1.84 \mathrm{E}-04$ & 0.25 & $1.21 \mathrm{E}-02$ & 0.97 & $1.30 \mathrm{E}-02$ & -0.95 & $1.99 \mathrm{E}-01$ & -0.49 \\
			\hline
\end{tabular}
\end{table}
\begin{table}[!h]
\caption{Error of the FEM-ImEx scheme at $t^n=1.0$, with varying $h$ but fixed $\Delta t=0.01$ and $\epsilon=0.0085$.}
\label{table6}
\begin{tabular}{c|c|c|c|c|c|c|c|c}
\hline\hline$h$ & $\left\|e_{2,\textbf{u}}^{h,n}\right\|_0$ & Rate & $\left\|e_{2,\textbf{u}}^{h,n}\right\|_1$ & Rate & $\left\|e_{2,\phi}^{h,n}\right\|_0$ & Rate & $\left\|e_{2,\phi}^{h,n}\right\|_1$ & Rate \\
\hline $1 / 2$ & $3.51 \mathrm{E}-02$ & & $4.40 \mathrm{E}-01$ & & $1.34 \mathrm{E}-01$ & & $5.09 \mathrm{E}-01$ & \\
		$1 / 4$ & $9.03 \mathrm{E}-03$ & 2.0 & $1.09 \mathrm{E}-01$ & 1.1 & $7.70 \mathrm{E}-02$ & 0.80 & $3.56 \mathrm{E}-01$ & 0.52 \\
		$1 / 8$ & $2.26 \mathrm{E}-03$ & 2.0 & $9.59 \mathrm{E}-02$ & 1.1 & $6.07 \mathrm{E}-02$ & 0.34 & $3.42 \mathrm{E}-01$ & 0.055 \\
		$1 / 16$ & $7.78 \mathrm{E}-04$ & 1.5 & $4.75 \mathrm{E}-02$ & 1.0 & $4.27 \mathrm{E}-02$ & 0.51 & $3.52 \mathrm{E}-01$ & -0.04 \\
		$1 / 32$ & $4.93 \mathrm{E}-04$ & 0.66 & $2.41 \mathrm{E}-02$ & 0.98 & $3.55 \mathrm{E}-02$ & 0.27 & $3.46 \mathrm{E}-01$ & 0.025 \\
		$1 / 64$ & $4.23 \mathrm{E}-04$ & 0.22 & $1.27 \mathrm{E}-02$ & 0.92 & $3.78 \mathrm{E}-02$ & -0.091 & $3.47 \mathrm{E}-01$ & -0.0022 \\
		\hline
\end{tabular}
\end{table}
\begin{table}[!h]
\caption{Error and convergence order of the MsFEM-ImEx scheme at $t^n=1.0$, with fixed $\Delta t=0.01$, $\frac{\epsilon}{h}=0.32$ and $nsplit=32$.}
\label{table7}
\resizebox{\linewidth}{!}{ 
\begin{tabular}{c|c|c|c|c|c|c|c|c|c}
\hline\hline$h$ & $\epsilon$ & $\left\|e_{1,\textbf{u}}^{h,n}\right\|_0$ & Rate & $\left\|e_{1,\textbf{u}}^{h,n}\right\|_1$ & Rate & $\left\|e_{1,\phi}^{h,n}\right\|_0$ & Rate & $\left\|e_{1,\phi}^{h,n}\right\|_1$ & Rate \\
		\hline $1 / 8$ & 0.04 & $2.29 \mathrm{E}-03$ & & $9.59 \mathrm{E}-02$ & & $1.52 \mathrm{E}-02$ & & $2.22 \mathrm{E}-01$ & \\
		$1 / 16$ & 0.02 & $6.32 \mathrm{E}-04$ & 1.9 & $4.73 \mathrm{E}-02$ & 1.0 & $1.05 \mathrm{E}-02$ & 0.53 & $1.73 \mathrm{E}-01$ & 0.36 \\
		$1 / 32$ & 0.01 & $2.41 \mathrm{E}-04$ & 1.4 & $2.37 \mathrm{E}-02$ & 1.0 & $9.39 \mathrm{E}-03$ & 0.16 & $1.58 \mathrm{E}-01$ & 0.13 \\
		$1 / 64$ & 0.005 & $1.44 \mathrm{E}-04$ & 0.74 & $1.20 \mathrm{E}-02$ & 0.98 & $9.07 \mathrm{E}-03$ & 0.05 & $1.54 \mathrm{E}-01$ & 0.038 \\
		\hline
\end{tabular}
}
\end{table}
\begin{table}[!h]
	\caption{Convergence performance of the MsFEM-ImEx scheme at $t^n=1.0$, with varying $\Delta t$ but fixed $h=1/8$, $nsplit=8$ and $\epsilon=0.0085$.}
\label{table8}
\resizebox{\linewidth}{!}{                
\begin{tabular}{c|c|r|c|r|c|c|c|c}
\hline\hline$\Delta t$ & $\left\|\mathbf{u}_{1,\Delta t}^{h,n}-\mathbf{u}_{1,\frac{\Delta t}{2}}^{h,n}\right\|_{0}$ & $\rho_{\textbf{u},h,0}$ & $\left\|\mathbf{u}_{1,\Delta t}^{h,n}-\mathbf{u}_{1,\frac{\Delta t}{2}}^{h,n}\right\|_{1}$ & $\rho_{\textbf{u},h,1}$ & $\left\|\phi_{1,\Delta t}^{h,n}-\phi_{1,\frac{\Delta t}{2}}^{h,n}\right\|_{0}$ & $\rho_{\phi,h,0}$ & $\left\|\phi_{1,\Delta t}^{h,n}-\phi_{1,\frac{\Delta t}{2}}^{h,n}\right\|_{1}$ & $\rho_{\phi,h,1}$ \\
		\hline 0.1 & $2.38 \mathrm{E}-05$ & 1.8 & $2.37 \mathrm{E}-04$ & 1.0 & $2.66 \mathrm{E}-03$ & 1.9 & $1.15 \mathrm{E}-02$ & 1.9 \\
		0.05 & $1.29 \mathrm{E}-05$ & 2.0 & $1.29 \mathrm{E}-04$ & 2.0 & $1.40 \mathrm{E}-03$ & 2.0 & $5.98 \mathrm{E}-03$ & 2.0 \\
		0.025 & $6.46 \mathrm{E}-06$ & 2.0 & $6.45 \mathrm{E}-05$ & 2.0 & $6.95 \mathrm{E}-04$ & 2.0 & $3.00 \mathrm{E}-03$ & 2.0 \\
		0.0125 & $3.23 \mathrm{E}-06$ & 2.0 & $3.23 \mathrm{E}-05$ & 2.0 & $3.46 \mathrm{E}-04$ & 2.0 & $1.50 \mathrm{E}-03$ & 2.0 \\
		0.00625 & $1.62 \mathrm{E}-06$ & & $1.61 \mathrm{E}-05$ & & $1.73 \mathrm{E}-04$ & & $7.49 \mathrm{E}-04$ & \\
		\hline
	\end{tabular}
}
\end{table}
Firstly, we examine the convergence order of the MsFEM-ImEx scheme in space. By varying the mesh size $h$ and keeping $\Delta t=0.01,$ $h_{fine}=1/2048$, and $\epsilon=0.0085$ fixed, we obtain the errors and convergence orders in Table \ref{table5}. When $h\textgreater \epsilon$, the convergence order for the $L^2$ norm error of $\phi$ gradually decreases from $O(h^2)$ to $O(h^{-1}), $ and the convergence order for the $H^1$ norm decreases from $O(h)$ to $O(h^{-1/2})$. 
When transferring information through the interface $\Gamma$, influenced by the Darcy region, with the refinement of the mesh, there is a decrease in the convergence order of the $L^2$ norm error for $\textbf{u},$ while there is essentially no significant impact on the $H^1$ norm error of $\textbf{u}$.

Then, we present the errors for the FEM-ImEx scheme in table \ref{table6}. Comparing Table \ref{table5}(MsFEM-ImEx) and Table \ref{table6}(FEM-ImEx), it is evident that the $L^2$ and $H^1$ errors of $\phi$ computed using the MsFEM-ImEx scheme is smaller.
The $L^2$ norm error for $\textbf{u}$ is the same in both schemes when the mesh is relatively coarse. However, with mesh refinement, the MsFEM-ImEx scheme yields a smaller $L^2$ error for $\textbf{u}$. The $H^1$ error for $\textbf{u}$ remains essentially consistent in both schemes as the grid size changes.

Next, from Table \ref{table7}, it can be observed that with grid refinement, the $L^2$ and $H^1$ errors of $\phi$ remains essentially unchanged, consistent with theoretical analysis. However, $\textbf{u}$ is influenced by the Darcy region, leading to a decrease in the convergence order.

Finally, the data in Table \ref{table8} indicates that the convergence orders for all variables in time are $O(t)$.

\begin{example}
We consider a modified square cavity flow model \cite{LLCG}. Let $\Omega_{p}=[0,1]\times[0.25,1]$, $\Omega_{f}=[0,1]\times[1,1.25]$ and $\Gamma=(0,1)\times\{1\}$.
In the Darcy region, we set $f_{p}=0$, the initial value $\phi(0)=0$ and $\phi=0$ on $\Gamma_{p}$. In addition, we assume $\mathbb{K}^{\epsilon}(\textbf{x})=K^{\epsilon}(\textbf{x})I$, where 
$$ K^{\epsilon}(\textbf{x})=
\frac{2.0+P\sin(\frac{2\pi x}{\epsilon})}{2.0+P\cos(\frac{2\pi y}{\epsilon})}  +  
\frac{2.0+P\sin(\frac{2\pi y}{\epsilon})}{2.0+P\cos(\frac{2\pi x}{\epsilon})},\ (P=1.5).$$
Then we set the right-hand side, the initial value, and boundary conditions of Stokes region as follows:
\begin{align*}
\textbf{f}_{f}&=[0,0],\\
\textbf{u}(0) &= [\sin(\pi x),0], \\
\textbf{u} &= [\sin(\pi x),0] \quad \text{on} \ (0,1)\times  \{1.25\},\\
\textbf{u} &= [0,0] \quad \text{on}\  \{0\}\times (1,1.25) \cup  \{1\}\times (1,1.25).
\end{align*}
We set T = 1, and all physical parameters $S_{0},\nu ,\alpha, g$ equal to 1.
\end{example}
\begin{figure}[h]
\centering
\subfigure[Reference] {\includegraphics[width=0.40\linewidth]{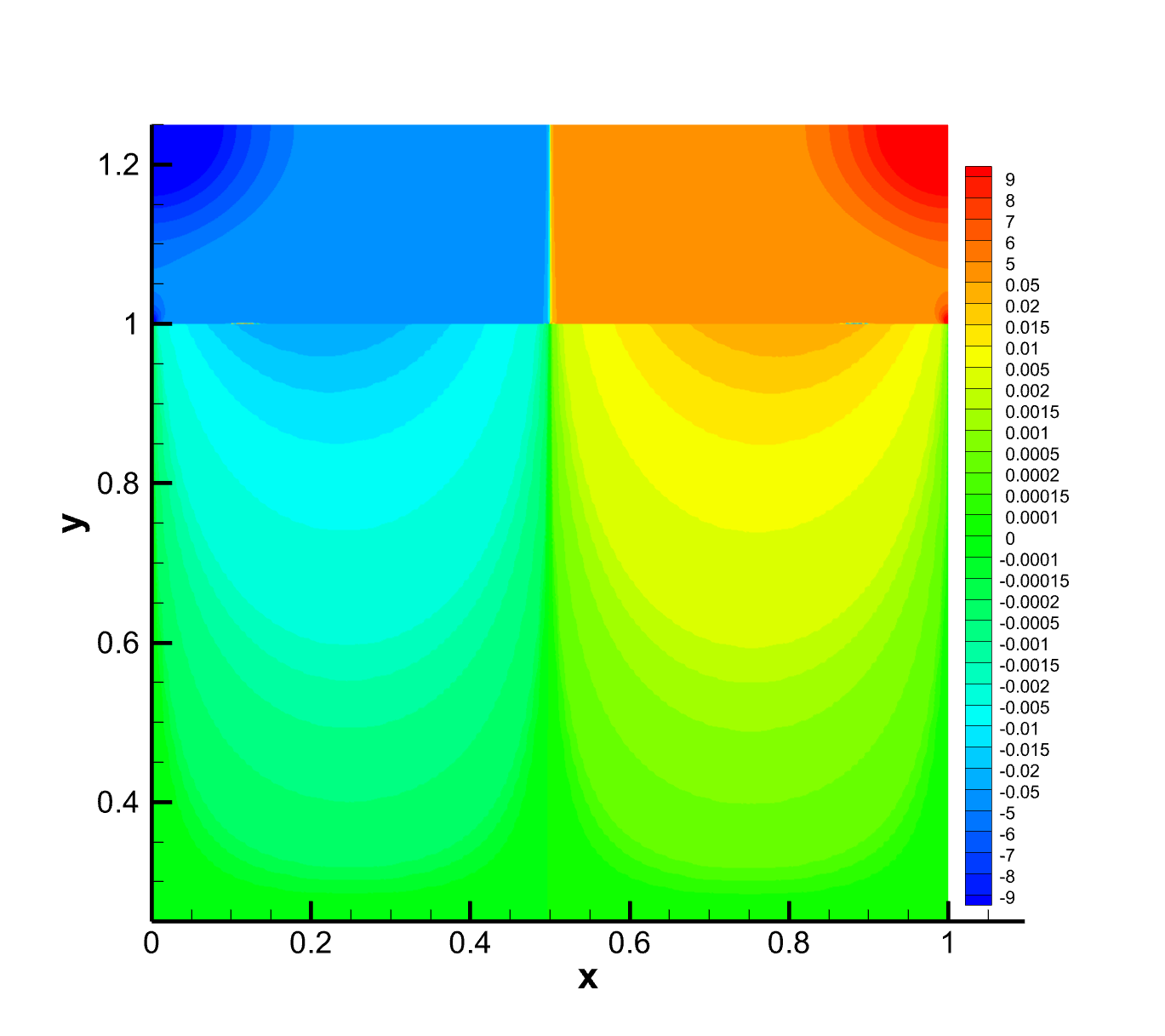}}
\subfigure[MsFEM-ImEx]{\includegraphics[width=0.40\linewidth]{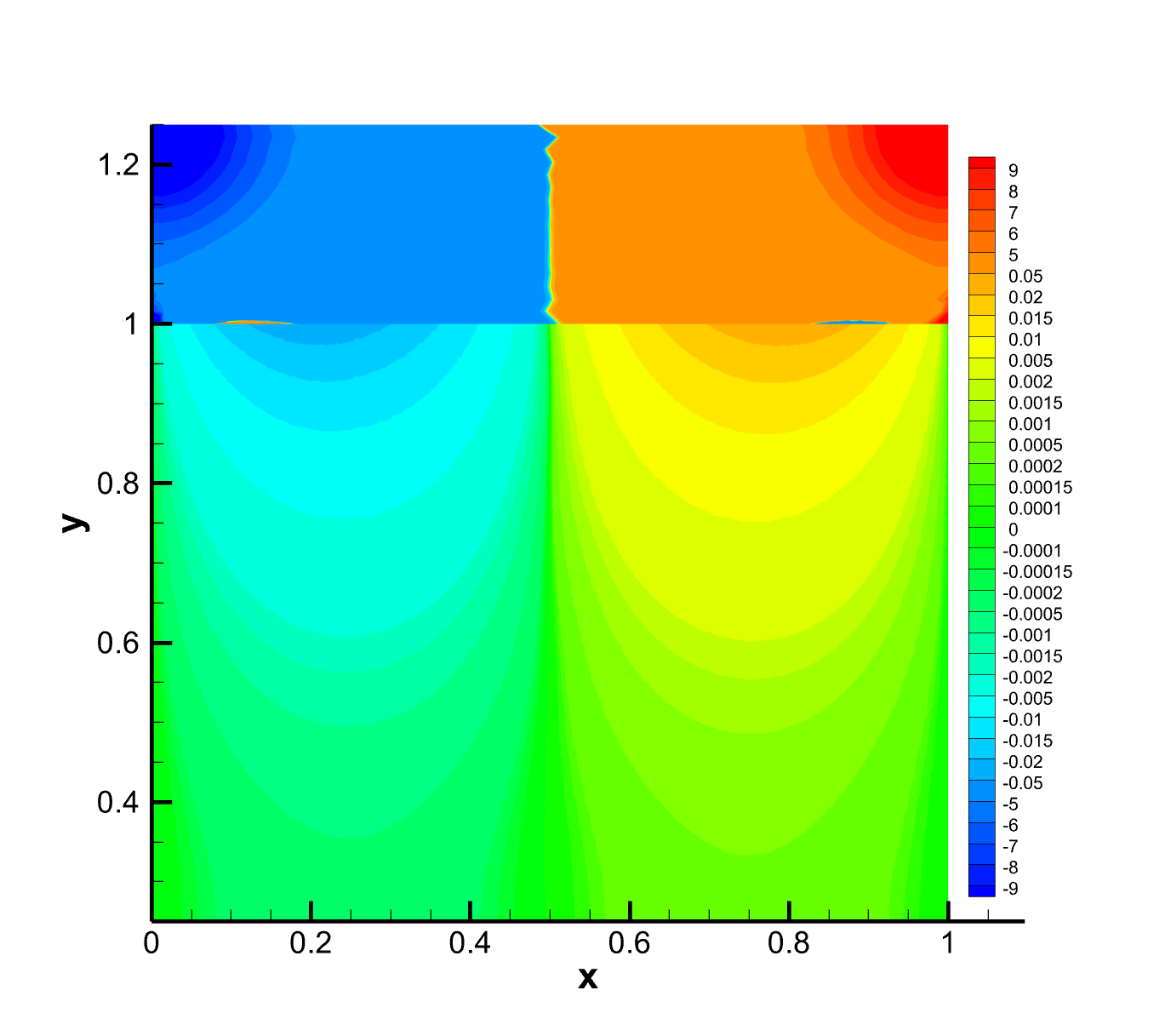}}
\caption{For $\epsilon=0.008$, the pressure distribution at $t=1.0$ for the reference solution(a), the pressure distribution at $t=1.0$, $h=1/64$ with MsFEM-ImEx scheme(b).} 
\label{fig:pressure}  
\end{figure}
This example illustrates the effectiveness of our proposed algorithm in addressing problems. We choose $\epsilon=0.008$ in this example.
Considering the lack of a true solution, we plot the pressure distribution of the reference solution at $t=1.0$ for the entire region, as shown in Figure \ref{fig:pressure}(a).
At the same time, on a triangular mesh with a grid size of $h=1/64,$ we fixed the time step $\Delta t=0.01$ and $h_{fine}=1/2048$. Then, we utilize Algorithm \ref{alg:1}(MsFEM-ImEx) proposed in this paper to solve this example. The pressure distribution at $t=1.0$ is depicted in Figure \ref{fig:pressure}(b).
By comparing the two pressure distribution figures in Figure \ref{fig:pressure}, it is apparent that their pressure distributions are essentially consistent throughout the entire region. This strongly validates the effectiveness of the Algorithm \ref{alg:1} (MsFEM-ImEx) proposed in this paper.

\section{Conclusion}
In this paper, based on the multiscale finite element method, we propose an algorithm for solving the non-stationary Stokes-Darcy model, where the Darcy region exhibits multiscale characteristics.
We demonstrate the stability of the algorithm and provide an error analysis.
Last, we validate the accuracy and effectiveness of the algorithm through numerical experiments.

\end{document}